\documentclass[11pt]{article}
\usepackage[utf8]{inputenc}
\usepackage[english]{babel}
\usepackage{amsmath,amsfonts,amssymb,amsthm,mathtools, enumitem}
\usepackage{tikz-cd, tcolorbox} 
\usepackage[nohead, nomarginpar, margin=0.75in]{geometry}
\usepackage[hidelinks]{hyperref}
\hypersetup{colorlinks,citecolor=black,filecolor=black,
    linkcolor=black,urlcolor=black}
\usepackage{cleveref,xurl}

\newtheorem{theorem}{Theorem}[section]
\newtheorem{corollary}[theorem]{Corollary}
\newtheorem{lemma}{Lemma}
\theoremstyle{definition}
\newtheorem{definition}[theorem]{Definition}
\newtheorem{proposition}[theorem]{Proposition}
\newtheorem{remark}[theorem]{Remark}
\newtheorem{example}[theorem]{Example}

\newcommand\K{\mathbf{K}}
\newcommand\T{\mathbf{T}}
\newcommand\U{\mathbf{U}}
\newcommand\com[1]{}
\newcommand\C{{\mathbb C}}

\newcommand\op[1]{\mathop{\rm #1}\nolimits}
\newcommand\p{\partial}
\newcommand\R{{\mathbb R}}
\newcommand\Z{{\mathbb Z}}

\newcommand\E{\mathcal{E}}
\newcommand\g{{\mathfrak g}}

\title{Differential Invariants of Carrollian Spacetimes}
\author{Boris Kruglikov, Eivind Schneider, Wijnand Steneker}

\date{}

\begin{document}
\maketitle

 \begin{abstract}
We compute invariants of Carrollian spacetimes, deriving them 
from the geometry of the screen bundle. 
For generic Carrollian structures, we specify how to generate the entire algebra of
differential invariants, with an emphasis on dimension 3, which has special
physical relevance.
Then, in the framework of jet-spaces, we compute the numerology behind these invariants: the Hilbert and Poincar\'e functions
that govern their numbers according to order.
Finally, we compute the Spencer cohomology behind the Carrollian geometry
that, in particular, contains the spaces of intrinsic torsion and intrinsic curvature,
which are fundamental invariants, important in the equivalence problem and symmetry analysis. Thus, we also discuss symmetry sizes of Carrollian spacetimes.
 \end{abstract}

\section{Introduction}

A \textit{Carrollian spacetime} is a triple $(M,g,K)$, where 
$M$ is a $(d+1)$-dimensional manifold and $g$ is a degenerate quadratic form
on $TM$ with kernel spanned by a nonvanishing vector field $K$; we also assume
$g$ nonnegative, i.e.\ $g(X,X)\geq0$ for $X\in TM$.
Degenerate metrics $g$ of rank $d$ arise as the ultra-relativistic limits 
of Lorentzian metrics when the speed of light $c\to0$, that is 
when the null cone of $g$ collapses to a line, and the null vector $K$ provides a time-scale on that line, see \cite{CP,Figueroa2}. We will assume that $d \geq 2$.

In general relativity, Carrollian spacetimes occur as null hypersurfaces with a 
distinguished null vector field. 
An important example is provided by a {\it Killing horizon}, which is 
a hypersurface $M\subset(\hat{M},\hat{g})$ in a Lorentzian manifold of dimension $(d+2)$
with a Killing vector field $K$ that is tangent to $M$ and is null along it. 
In this case $(M,\hat{g}|_M,K|_M)$ is a Carrollian spacetime. The case when $d=2$ has special physical relevance since 4-dimensional spacetimes can play the role of $\hat{M}$.  If $d\|K\|_g^2 \neq 0$ on $M$, then the hypersurface $M$ determined by the condition $\|K\|_g^2=0$ is a smooth, isolated Killing horizon. However, there exist Lorentzian spacetimes that are foliated by Killing horizons, such as Kundt spacetimes \cite{PLJ,MS}. 


Other examples come from asymptotic geometries \cite{ABL,Gibbons,BMN}.
The event horizons are null hypersurfaces $M\subset\hat{M}$ that often differ 
from Killing horizons. There may be no canonical choice of a null vector 
field $K$ on a null hypersurface. For generic non-expanding horizons, 
\cite{ABL} shows how one can choose $K$ uniquely, up to a constant factor. 
In more general settings, the null vector field can be introduced for generic 
null hypersurfaces via differential invariants: 
if $J$ is such an invariant with $dJ$ non-vanishing on the null line field,
then it trivializes this line field. 

Yet another example of Carrollian structures comes from conformal geometry: 
if $(N,[q])$ is a conformal structure, its bundle of scales 
$M=N\times\R(t)\stackrel{\pi}\to N$ has a natural degenerate quadric 
$g=e^t\pi^*q$ and the time-scale $K=\p_t$. 
The Fefferman--Graham construction realizes $M\subset\hat{M}$ as a null hypersurface
in a vacuum Einstein space on a (finite or infinite, depending on the parity of $d$)
jet-level along $M$ \cite{CG}. 

Invariants of Carrollian spacetimes and, more generally, null hypersurfaces
were extensively discussed in the literature \cite{NR,Duval,Figueroa}.
The classical Cartan--Karlhede approach is not directly applicable 
since these structures have infinite type (and so the corresponding Cartan
bundle is infinite-dimensional). Therefore, additional ingredients (like a choice
of compatible connection or Ehresmann connection) are often introduced 
\cite{Ciambelli,Herfray}.
Others employ extrinsic geometry through the restriction of the Cartan bundle
of $\hat{M}$ to $M$ \cite{Blitz,BMN}. Several papers address the question of 
unique realization of Carrollian structures 
as hypersurfaces in Lorentzian spacetimes $M\subset\hat{M}$, cf.\ \cite{Morand}. 
This is a vast subject of near-horizon geometry \cite{ABL,CST}. 

In this paper, we will pursue the purely intrinsic viewpoint 
and compute differential invariants of Carrollian spacetimes without additional geometric data.
While structures that are close to
integrable are important for applications and are often subject to simpler
computations of differential invariants (for instance, if the intrinsic torsion vanishes), 
we will focus on generic spacetimes. From physical viewpoints, these may
be interpreted through rigged connections and strong gravity distortion \cite{CP}.
In this general setting, we find generators for the entire algebra of 
local differential invariants. For extremal spacetimes and near horizon geometry,
this approach was previously realized in \cite{KS,MS} and we now apply it
to the classification of Carrollian structures.

In this context, the screen bundle $\U = TM/\langle K \rangle$ plays an important role. In particular, the quadratic form $g$, along with the intrinsic torsion $q_1=L_K g$ (and its higher-order analogs $q_l$), 
is a tensor field on $\U$. In section \ref{S2}, we show that,
provided the intrinsic torsion is nondegenerate,
there is a canonical flat connection on $\U$. When $\dim M=3$, it is given by the two conditions 
 \[ 
\nabla g=0, \qquad \nabla q_1 = \alpha \otimes g + \beta \otimes q_1,
 \]
for some 1-forms $\alpha,\beta$ that are uniquely determined, and it extends to the case of general $d$.  
Moreover, since $g$ is nondegenerate on $\U$, we can use it to raise and lower indices for tensors on this bundle. 
This lets us construct a sequence of operator fields $Q_l \in \op{End}(\U)$ from $q_l$. Similar to the classical approach with 
scalar polynomial curvature invariants (SPI) in general relativity,
these give a family of invariants $I_\sigma=\op{tr}(Q^\sigma)$,
enumerated by multi-indices $\sigma=(i_1,\dots,i_m)$ of arbitrary but finite
length, where $Q^\sigma=(Q_1)^{i_1}\cdots(Q_m)^{i_m}$.  

These scalar (polynomial in jets) invariants generate the entire algebra of 
absolute differential invariants as follows.
Choose $(d+1)$ functionally independent invariants $I^0,\dots,I^d$ from
this family, which means the indicated invariants are functionally independent 
almost everywhere for generic Carrollian data $(g,K)$. 
They can be considered as local coordinates on $M$, and we can
express the Carrollian structure in terms of those:
 \[ 
g=\sum_{i,j=1}^d G_{ij}(I^0,\dots,I^d) dI^i dI^j,\qquad 
K=\sum_{i=0}^d H^i\p_{I^i}.
 \]
Writing $(g,K)$ in these coordinates amounts to bringing the Carrollian structure to normal form, 
thus solving the equivalence problem and resulting in a complete set of generators $(I^i,G_{ij},H^i,\p_{I^i})$ for 
the field of rational differential invariants. 

While this, in principle, solves the equivalence problem and provides us with all differential invariants, the actual expressions are involved. 
Moreover, the generators are not of minimal order,
and they neither indicate the number of differential invariants according to their order nor how to obtain those invariants 
(which turn out to be some complicated syzygies of the generators).
In Section \ref{S3}, we therefore conduct a more detailed analysis of the differential invariants. 
We provide a count of independent differential invariants
(via Hilbert and Poincar\'e functions \cite{BK}), 
and we find a transcendence basis for the field of rational second-order differential invariants when $\dim M=3$. The computations are done in the framework of jet bundles, which we will
briefly overview (within this framework, $I_\sigma$ are functions on
the space of jets, and $dI^i$ are horizontal differentials,
while $\p_{I^i}$ are total derivatives).  

Afterwards, in Section \ref{S4} we compute the Spencer cohomology
of the Carrollian structure (viewed as a $G$-structure). Intrinsic torsion naturally 
takes values in this cohomology, 
and is well studied \cite{Figueroa,Figueroa1}, 
so we discuss the second crucial ingredient, 
the intrinsic curvature. In the literature, the intrinsic curvature is typically defined 
for $G$-structures with vanishing intrinsic torsion. Here we define it also for non-vanishing (more precisely, nondegenerate) intrinsic torsion. 
The intrinsic curvature, like intrinsic torsion, 
is a fundamental invariant and hence is related
to the differential invariants studied in Section \ref{S3}. The nature of this relationship, however, is involved,
especially for generic Carrollian spacetimes.

\section{Geometry of the screen bundle}\label{S2}

A null hypersurface underlying the Carrollian structure is the triple
$(M,g,\K)$, where $\K=\langle K\rangle$
is the line bundle of kernels of the degenerate quadratic form $g$. We consider it intrinsically, without encoding any ambient space, and so the word {\em hypersurface} is nominal (alternative term: null manifold).

The {\em screen bundle} over a null hypersurface is $\U=TM/\K$. (For Bargmannian 
structures, the screen bundle is a bundle over spacetime, cf.\ \cite{BM,Figueroa1,FLTC},
but it is fully justified to have it over the null or Carrollian space as well.)
Note that $g$ is naturally a metric on this bundle $\U$, 
which we denote by the same symbol (the quadratic form on $M$
is recovered from it by pull-back).

\subsection{Intrinsic torsion and higher order endomorphisms} \label{S21}

Define a tensor field on $M$ via the Lie derivative by $q_1=L_Kg$;
this is called {\em intrinsic torsion} in the literature \cite{BM,Figueroa}, 
and we explain this naming in Section \ref{S4}. It is a degenerate quadratic form with the same kernel $\K$
and thus also descends to the screen bundle $\U$. 
Equivalently, we obtain a symmetric endomorphism $Q_1$ of $\U$ 
by raising an index with $g$:
 $$
g(Q_1X,Y)=q_1(X,Y),\qquad X,Y\in\U.
 $$
 \begin{lemma}
The tensor field $q_1$ (or $Q_1$) is the obstruction to $g$ being locally projectable along $K$.
 \end{lemma}

This lemma is obvious, however we would like to stress its local nature.
If the foliation of $M$ by $K$ is a global fibration over a manifold $U$, then
$g$ is a pullback of a metric on $U$, and the entire structure is reduced
to a Riemannian metric in dimension $(d-1)$. 

Similarly, if 
$Q_1=\op{diag}(\lambda,\dots,\lambda)$ is proportional to the identity,
then $g$ descends to $U$ as a conformal structure and 
conversely $M$ can be identified as a bundle of scales,
used in the Fefferman--Graham construction for $(U,[g])$, see \cite{CG} for
a tractor description.

 \begin{definition}\label{ndgq1}
We call the intrinsic torsion $q_1$ {\em nondegenerate} if $Q_1$ has a
simple spectrum, i.e.\ is conjugate to $\op{diag}(\lambda_1,\dots,\lambda_d)$
with $\lambda_i\neq\lambda_j$ for $i\neq j$.
We call it {\em strongly nondegenerate} if, in addition, all $\lambda_i\neq0$.
 \end{definition}

From the operator $Q_1$, we derive scalar invariants via traces of its powers: 
 $$
I^k_1=\op{tr}\bigl((Q_1)^k\bigr).
 $$
These become dependent for $k>d$ because of the Cayley--Hamilton 
theorem $p_{Q_1}(Q_1)=0$, where 
$p_{Q_1}(\lambda)=\op{det}(Q_1-\lambda\cdot{\bf1})=
(-\lambda)^d+c_{d-1}(-\lambda)^{d-1}+\dots-c_1\lambda+c_0$, whence
$I_1^d=c_{d-1}I_1^{d-1}-\dots+(-1)^dc_1I_1^1-(-1)^dc_0d$ and similarly for $k>d$:
$I_1^{d+1}=c_{d-1}I_1^d-c_{d-2}I_1^{d-1}+\dots+(-1)^dc_1I_1^2-(-1)^dc_0I_1^1
=(c_{d-1}^2-c_{d-2})I_1^{d-2}-\dots+(-1)^d(c_1c_{d-1}-c_0)I_1^1-(-1)^dc_0c_{d-1}d$, etc.
 
Similarly, we get higher order tensors $q_l\in S^2\U^*$ and endomorphisms
$Q_l\in\op{End}(\U)$ via
 $$
q_l= L_K^lg = \underbrace{(L_K\circ\cdots\circ L_K)}_{l\, \mathrm{times}}g,\qquad g(Q_lX,Y)=q_l(X,Y).
 $$
Thus, we obtain higher order invariants
$I_l^k=\op{tr}\bigl((Q_l)^k\bigr)$ for $l>0$, $k<d$.
There are also traces of mixed product, as indicated in the introduction,
and they may give more invariants in addition to $I_l^k$.

Actually, in general, the spectrum of $Q_1$ is simple, so we have $d$ different
eigenvalues $\lambda_1,\dots,\lambda_d$ and corresponding eigenvectors 
$v_1,\dots,v_d$.  We assume that they are normalized $\|v_i\|_g=1$, which defines them up to $\pm$. 
Thus 
$q_l(v_i,v_j)$ give more independent invariants compared to $I_l^k$ for $l>1$.
Note that these invariants are local, as they depend on numeration of 
eigenvalues and express via formulae in radicals.

One way to resolve this problem is by averaging over the action of the relevant symmetry group of the roots 
$S_d\ltimes\Z_2^d$, which yields differential invariants that are rational in jets. In Section \ref{S3.3}, we use another approach to construct global invariants and readily identify syzygies among them.

\subsection{Remark on null hypersurfaces}

For null hypersurfaces, considered intrinsically, in contrast to 
Carrollian spacetimes, we do not have a distinguished null vector field $K$, 
but only its span $\K=\op{Ker}(g)$. One possibility is to 
select a section $K$ of $\K$ via normalization $dI(K)=1$ for some
differential invariant, for instance $I_1^1$, thus obtaining a Carrollian 
structure. Another approach to obtain invariants is as follows.

 \begin{lemma}
Let $(M,g,\K)$ be a null hypersurface. For a null vector field $X\in\Gamma(\K)$ 
the Lie derivative $L_Xg$ is tensorial in $X$, i.e.,\ $L_{fX}g=fL_Xg$
for all $f\in C^{\infty}(M)$.
 \end{lemma}

 \begin{proof}
In addition to $X$, choose vector fields $Y,Z$. Then we get:
 \begin{equation*}
 \begin{split}
(L_{fX}g)(Y,Z) & = L_{fX}(g(Y,Z)) -g(L_{fX}Y,Z) -g(Y,L_{fX}Z) \\ 
& =  fL_{X}(g(Y,Z)) -g(fL_XY-Y(f)X,Z)  -g(Y,fL_XZ-Z(f)X) = f\cdot(L_Xg)(Y,Z).
 \end{split}
 \end{equation*}
Since $Y,Z$ are arbitrary, the claim follows.
 \end{proof}

Thus the corresponding operators $Q_l\in\op{End}(\U)$ are defined all up to the
same factor. We can extract the invariants of $Q_1$ as either the projective
quantity $[\lambda_1:\dots:\lambda_d]$ or by normalization $\sum|\lambda_i|^2=1$
or $\max|\lambda_i|=1$. Then we get $(d-1)$ invariants of order 1 from this
construction (one less than for Carrollian) but the spectrum of $Q_l$ for $l>1$ 
gives the same type invariants as for Carrollian spacetimes.

Alternatively, we can construct rational invariants via traces as follows:
 \[
\op{tr}\bigl((Q_1)^2\bigr)/\op{tr}(Q_1)^2,\ 
\op{tr}\bigl((Q_1)^3\bigr)/\op{tr}(Q_1)^3,\ \dots,
\op{tr}\bigl((Q_1)^d\bigr)/\op{tr}(Q_1)^d,
 \]
and similarly for $Q_l$ with $l>1$. 

Thus, we see a close relationship between the differential invariants of null hypersurfaces 
and  Carrollian spacetimes, and between the corresponding equivalence problems.

\subsection{A canonical connection on the screen bundle}

Under nondegeneracy assumption, 
we shall define a canonical connection $\nabla:\Gamma(\U)\to\Omega^1(M)\otimes\Gamma(\U)$.

 \begin{proposition}\label{cancon}
Suppose that the intrinsic torsion $q_1$ is nondegenerate. 
Then there exists a unique connection $\nabla$ on $\U$ satisfying 
 \begin{equation}\label{U-connection}
\nabla g=0, \qquad \nabla Q_1 = \sum_{i=0}^{d-1}\alpha_i\otimes (Q_1)^i
 \end{equation}
for some (uniquely determined) 1-forms $\alpha_i\in\Omega^1(M)$.
 \end{proposition}

\begin{proof}
Let $\nabla^0$ be a connection preserving $g$ on $\U$ (such always exists). 
Any other connection on the bundle $\U$ is given by $\nabla=\nabla^0+A$, 
where $A\in\Omega^1(M,\mathfrak{so}(\U))$. 

Since $Q_1$ is symmetric and $\nabla^0$ is metric compatible, we have that 
$\nabla_X^0Q_1=[b_{ij}]$ is self-adjoint; that is, $b_{ji}=b_{ij}$ 
in an orthonormal basis of $\U$. As such, we choose the normalized eigenbasis 
of $Q_1$, in which $Q_1=\op{diag}(\lambda_1,\dots,\lambda_d)$, whence
$(Q_1)^k=\op{diag}(\lambda_1^k,\dots,\lambda_d^k)$. 
Note that for any $X\in TM$, $A_X=[a_{ij}]$ is skew-symmetric in such a basis:
$a_{ji}=-a_{ij}$.
The second equation in \eqref{U-connection} becomes
 \begin{equation}\label{U-conn2}
[A_X,Q_1]=-\nabla^0Q_1+\sum_{i=0}^{d-1}\alpha_i(X)\otimes (Q_1)^i
 \end{equation}
where the summation on the right is represented by the Vandermonde matrix
$[\lambda_j^i]_{d\times d}$ times the column-vector of $\alpha_i(X)$.
By the assumption $\lambda_i\neq\lambda_j$ for $i\neq j$, this 
can achieve any diagonal value. In particular, we can choose 1-forms $\alpha_i$
such that the right hand-side of \eqref{U-conn2} is traceless. 
Note that $\op{tr}([A,B]^k)=0$ for any $k$, so this is a necessary condition
to resolve \eqref{U-conn2}.

The right-hand side of \eqref{U-conn2} is symmetric; denote its entries by 
$-b_{ij}^0=-b_{ij}+\frac1d(\sum b_{kk})\delta_{ij}$.
The left-hand side of \eqref{U-conn2} is also symmetric, since it is a commutator 
of a skew-symmetric and a symmetric matrix. 

Since $Q_1$ is nondegenerate, this can be uniquely solved.
Indeed, the formula $a_{ij}=\frac{b_{ij}^0}{\lambda_i-\lambda_j}$
gives the unique solution to this equation. 
 \end{proof}

 \begin{remark}
The connection $\nabla$ constructed in Proposition \ref{cancon} depends only 
on the null hypersurface structure underlying the Carrollian spacetime. 
Indeed, a change of the null field $K\mapsto f\cdot K$
results in the change $Q_1\mapsto f\cdot Q_1$, so equation \eqref{U-connection}
still holds with modified forms $\alpha_i$. Thus the same $\nabla$, obtained
for $(M,g,K)$, is also a canonical connection for the Carrollian structure $(M,g,f\cdot K)$.
 \end{remark}

The 1-forms $\alpha_i$ can be expressed through the first and second order
differential invariants $I_1^i,I_2^j$. For instance, in the case $d=2$
 $$
\alpha_0=\frac{I_1^2dI_1^1-I_1^1dI_1^2}{2I_1^2-(I_1^1)^2},\quad 
\alpha_1=\frac{dI_1^2-I_1^1dI_1^1}{2I_1^2-(I_1^1)^2}.
 $$
Next, we show that the canonical connection $\nabla$ is actually flat. 

 \begin{theorem}
The curvature $R^{\nabla}\in\Omega^2(M,\mathrm{End}(\U))$
of the connection $\nabla$ on the screen bundle vanishes identically.
 \end{theorem}

\begin{proof}
The connection $\nabla$ defines an exterior derivative $d_{\nabla}:\Omega^{\bullet}(M, 
\mathrm{End}(\U)) \rightarrow \Omega^{\bullet + 1}(M, \mathrm{End}(\U))$. The curvature 
$R^{\nabla}$ measures the extent to which the vanishing of $d_{\nabla}^2$ fails, more precisely one has $d_{\nabla}^2(\omega) = R^{\nabla} \wedge\, \omega$ for all $\omega \in \Omega^{\bullet}(M, \mathrm{End}(\U))$. 

In view of the defining property \eqref{U-connection} of $\nabla$, we compute that
\begin{equation}
    \begin{split}
        R^\nabla_{X,Y} Q_1 
& = \sum_{i=0}^{d-1} d\alpha_i(X,Y)\cdot(Q_1)^i 
+ \sum_{i=0}^{d-1}\sum_{j=1}^{d-1}j\,(\alpha_i\wedge\alpha_j)(X,Y)\cdot(Q_1)^{i+j-1}.
    \end{split}
\end{equation}
We now show that $\nabla$ is flat by imposing the compatibility 
conditions $d^2 I_1^k = 0$ for $k = 1, \dots, d$. Consider the 
orthonormal eigenbasis $v_l$ for $Q_1$ and denote its eigenvalues by $
\lambda_l$. In particular, $R^{\nabla}_{X,Y}Q_1$ is diagonal in this 
eigenbasis. 

Imposing the compatibility conditions $d^2 = 0$ on the trace invariants yields
 \begin{equation}
0 = d^2I_1^k = d^2 \mathrm{tr}\bigl((Q_1)^k\bigr) 
= \mathrm{tr}\bigl(d_{\nabla}^2(Q_1)^k\bigr) 
= \mathrm{tr}\bigl(R^{\nabla} (Q_1)^k\bigr).
 \end{equation}
The covariant derivative $\nabla Q_1$ is polynomial in $Q_1$, which implies 
$R^{\nabla}(Q_1)^k = k\,(Q_1)^{k-1}R^{\nabla}Q_1$. 
By using $I_1^k = \mathrm{tr}\bigl((Q_1)^k\bigr) = \sum_l \lambda_l^k$, 
we obtain for $X,Y \in TM$ that 
 \begin{equation}
0 = k^{-1}\, d^2I_1^k(X,Y) = k^{-1}\, \mathrm{tr}\bigl(R_{X,Y}^{\nabla}(Q_1)^k\bigr) 
= \mathrm{tr}\bigl((Q_1)^{k-1}R_{X,Y}^{\nabla}Q_1\bigr) 
= \sum_{l = 1}^d \lambda_l^{k-1} g(v_l,(R_{X,Y}^{\nabla}Q_1)(v_l)).
 \end{equation}
By considering this equation for all $k = 1, \dots, d$, we get the action of the Vandermonde 
matrix $[\lambda_l^{k-1}]_{d\times d}$ on the diagonal of $R^{\nabla}_{X,Y}Q_1$. Nondegeneracy of the Vandermonde matrix is guaranteed by the assumption that $Q_1$ has simple spectrum. Consequently, the diagonal vanishes, and we obtain that $R_{X,Y}^{\nabla}Q_1 = 0$.

The eigenspaces of $Q_1$ are one-dimensional by simplicity of its spectrum. As $Q_1$ and $R^{\nabla}_{X,Y}$ commute, we have that 
$R^{\nabla}_{X,Y}$ preserves these one-dimensional eigenspaces, and so is 
also diagonal in this eigenbasis. In terms of the orthonormal eigenbasis $v_l$, the matrix $R_{X,Y}^{\nabla}$ is both skew-symmetric and diagonal which implies its vanishing. We conclude that the connection is flat.
\end{proof}

\section{Scalar differential invariants}\label{S3}

In the context of Carrollian structures $(M,g,K)$, differential invariants are functions on $M$ that are defined in terms of the components of $(g,K)$ and their partial derivatives in such a way that their expressions are coordinate-independent. 
Such functions are physically relevant, as they are observer-independent. 
In this section, we introduce another framework for differential invariants through jet-spaces, which allows us to discuss their completeness and independence.

We have already introduced differential invariants  $I_l^i=\op{tr}\bigl((Q_l)^i\bigr)$. 
While $I_1^1,\dots,I_1^d$ form a transcendence basis for the field of rational differential invariants of first order, the picture is more complicated in higher order. In this section, we investigate differential invariants in more detail. We construct an invariant frame in two different ways; one local and one global. Then we use these to find a maximal set of independent second-order differential invariants, focusing on the most interesting case when $\dim M =3$, and we show that these can be used to generate a maximal independent set of any order. In the end, we count the number of algebraically independent differential invariants of any order for $\dim M \geq 3$ and summarize the result through the Poincaré function.

\subsection{A brief introduction to jet spaces} \label{sect:jets}

Let us introduce some terminology and notation related to jet spaces, as this provides a good framework for studying differential invariants, cf.\ \cite{KL1, O}. 

A Carrollian spacetime is a section $(g,K)$ of the fiber bundle 
\[\pi\colon S^2 T^*M \oplus TM \to M,\]
satisfying two additional properties:
\begin{enumerate}
    \item $g(K,\cdot)=0$,
    \item $\dim \ker g = 1$.
\end{enumerate}
The first condition defines an algebraic set $\E^0 \subset  S^2 T^*M \oplus TM$, while the second constrains us to an open subset of $\E^0$. Both $\E^0$ and the open subset project surjectively onto $M$. A choice $x^1, \dots, x^{d+1}$ of
local coordinates on $M$, induces
coordinates $u_{ij}=u_{ji}$ and $v^i$ on the fibers of the bundle 
$S^2 T^*M \oplus TM$, in which 
 \[
\E^0 = \{(x,u,v) : u_{ij} v^i=0, j=1,\dots, d+1\}.
 \] 
This corresponds to the relation $g_{ij}K^i=0$ on
the section $g=g_{ij}dx^idx^j,K=K^i \p_{x^i}$ 
(summation over $i,j=1,\dots,d+1$ and we impose $g_{ji} = g_{ij}$).

With this choice of coordinates, the $k$-th order jet bundle $\pi_k \colon J^k \pi \to M$ comes with canonical fiber coordinates $u_{ij,\sigma}, v^i_\sigma$, with $\sigma=(i_1, \dots, i_{d+1})$ being a multi-index with $i_s \geq 0$ and length $|\sigma| = i_1 + \cdots + i_{d+1} \leq k$. The section of $\pi$ given by $u_{ij} = g_{ij}(x), v^i =K^i(x)$ is naturally prolonged to a section of $\pi_k$, given by $u_{ij,\sigma} = \frac{\partial^{|\sigma|}}{\partial x^\sigma} g_{ij}(x),v^i_{\sigma} = \frac{\partial^{|\sigma|}}{\partial x^\sigma} K^i(x)$, which we refer to as the $k$-prolongation of the section $s=(g,K)$ and denote by $j^k s$. Another natural bundle structure that will be used in this paper is given by the projection 
\[\pi_{k,l}\colon J^k\pi \to J^l \pi, \quad k > l.\] 

We introduce the total derivative 
 \[ 
D_l = \partial_{x^l} + \sum_{|\sigma| = 0}^\infty \left( u_{ij,\sigma+1_l} \partial_{u_{ij,\sigma}} + v^i_{\sigma+1_l} \partial_{v^i,\sigma}\right):
C^{\infty}(J^{\infty})\to C^{\infty}(J^{\infty}).
 \]
The horizontal exterior derivative $\hat{d}: C^{\infty}(J^{\infty})\rightarrow\Omega^1(J^{\infty})$ acts on jet-functions by the formula
\begin{equation}
    \hat{d}f = \sum_{l = 1}^{d+1} D_l(f) dx^l.
\end{equation}  
Note that $\hat{d} f$ is not an arbitrary 1-form on $J^\infty \pi$, since it is a $C^\infty(J^\infty \pi)$-linear combination of $dx^1, \dots, dx^{d+1}$ only (no $du_{ij}$, $dv_i$, $du_{ij,l}$, etc.). Such 1-forms are called {\it horizontal}. 

If we restrict to the jet-prolongation of a section $s=(g,K) \in \Gamma(\pi)$, then the horizontal derivative agrees with the ordinary exterior derivative, i.e., $(\hat{d}f) \circ j^k s = d(f \circ j^k s)$. For example, $\hat{d}u_{ij} = u_{ij, l}\, dx^l$. Similar to the usual exterior derivative, the horizontal derivative extends to the map $\hat{d}: \Omega^{\bullet}(J^{\infty}) \rightarrow \Omega^{\bullet+1}(J^{\infty})$ satisfying 
the property $\hat{d}^2 = 0$ and the graded Leibniz rule.

If the functions $g_{ij}, K^i$ are smooth and satisfy $g_{ij}K^i=0$, they will also satisfy $\frac{\partial^{|\sigma|}}{\partial x^\sigma} (g_{ij} K^i)=0$ for any multi-index $\sigma$. For this reason, it makes sense to define the $k$-prolongation of $\E^0$ by 
 \begin{equation*}
\mathcal{E}^{k}:= \{D_{\sigma}(u_{ij} v^i)=0 : 0\leq|\sigma|\leq k\} \subset J^k\pi.
 \end{equation*}
Here $D_\sigma$ is the iterated total derivative, defined by composing total derivatives with respect to the $x^i$-coordinates according to the multi-index $\sigma$. It is readily seen that no compatibility conditions arise by prolonging the algebraic equation, i.e., the differential equation is formally integrable.

Any local diffeomorphism $\varphi$ on $M$ induces a local diffeomorphism $\hat \varphi$ on $S^2 T^* M \oplus TM$ by the standard contravariance/covariance properties. Similarly, a vector field $X$ on $M$ lifts canonically to a vector field $\hat X$ on $S^2 T^* M \oplus TM$. More specifically, if $X = a^i(x) \partial_{x^i}$, then the canonical lift is given by 
\begin{equation} \label{eq:Xhat0}
    \hat X = a^i \partial_{x^i}- \sum_{i\leq j}(u_{lj} a^l_{x^i}+u_{il} a^l_{x^j}) \partial_{u_{ij}} + a^i_{x^j} v^j \partial_{v^i}.
\end{equation}
A diffeomorphism on $S^2 T^* M \oplus TM$ transforms sections of the bundle, and therefore also their $k$-jets. This gives a canonical prolongation $\hat \varphi^{(k)}$ of $\varphi$ to $J^k \pi$, and a canonical prolongation of $\hat X$ to a vector field $\hat X^{(k)}$ on $J^k \pi$. For the expression of $\hat X^{(k)}$, we refer to formula (2) in \cite{KL1} or Theorem 4.16 in \cite{O}.  Due to the coordinate-independent definition of $\E= \{\E^i\}_{i=0}^\infty$, it is clear that $\hat \varphi^{(k)}$ preserves the submanifold $\E^k \subset J^k\pi$ for any diffeomorphism $\varphi$ on $M$, and that $\hat X^{(k)}$ is tangent to $\E^k$ for any vector field $X$ on $M$. 

\begin{definition}
    A differential invariant $I$ of order $k$ (of Carrollian structures) is a function on $\E^k$ that is invariant under prolongations of all local diffeomorphisms. It satisfies 
    \[ L_{\hat X^{(k)}} I = 0, \qquad \forall X \in \mathfrak{X}(M). \]
\end{definition}

 \begin{remark} \label{rk:jets}
When considering a specific section $s=(g,K) \in \Gamma(M)$,  the differential invariant $I$ restricts to a function $I \circ j^k s \in C^\infty(M)$. Often, the latter function is also referred to as a differential invariant (e.g.\ Section \ref{S2}). 
However, the jet language brings precision and new tools 
that we take advantage of. 
 \end{remark}

As we will see, there exist differential invariants of any order. However, they can all be generated by a finite number of differential invariants and invariant derivations. For differential invariants in a general setting, this idea goes back to Sophus Lie. 
We will use a modern global formulation due to \cite{KL2}. We have simplified the statement by making appropriate specifications for the application to Carrollian spacetimes: 

\begin{theorem}[\textbf{The global Lie-Tresse theorem} \cite{KL2}]\label{thm:GLT}
There exists a finite number of rational differential 
invariants $I_1, \dots, I_q$ and $d+1$ invariant derivations $\nabla_1, \dots, \nabla_{d+1}$, such that any rational 
differential invariant is of the following form for some rational function $F$:
 \[ 
F(I_i, \nabla_{j_1}(I_i), \nabla_{j_2}\nabla_{j_1} (I_i), \dots).
 \]
The field of rational differential invariants separates 
orbits in general position in $\mathcal{E}^{\infty}$.
\end{theorem}

The invariant derivations are operators of the form $\nabla_i = \alpha_i^j D_j$, satisfying $[\hat X^{(\infty)}, \nabla_i]=0$ for each $X \in \mathfrak{X}(M)$. They take invariants of order $k$ to invariants of order $k+1$, provided that $k$ is sufficiently large. They let us generate differential invariants of any order from invariants of some fixed order. The $d+1$ invariant derivations also act as an invariant frame. 



\subsection{An invariant frame and second-order invariants}\label{S3.2}

Let $\omega^1, \dots, \omega^{d} \in \Gamma(\U^*)$ 
denote $d$ independent eigenvectors of the operator $Q_1^* \in \op{End}(\U^*)$. These exist for a generic Carrollian structure. The eigenvectors are determined only up to multiplication by a function, which means that we may choose them such that $\|\omega^i\|_g=1$. The remaining ambiguity is a $\pm 1$ factor for each eigenvector in addition to their ordering, and we will consider a fixed but arbitrary choice of these. The eigenvectors $\omega^1, \dots, \omega^{d}$ can be considered as 1-forms on $M$ by pullback through the map $TM \to \U$. 

In the previous paragraph, everything is defined in terms of a specific section $(g,K)$. However, in the jet framework, the 1-forms $\omega^1, \dots \omega^d$ should be understood as horizontal 1-forms on $J^1 \pi$, and this is how we will treat them from now on. In a similar way, and in line with Remark \ref{rk:jets}, we will treat $g$ as a horizontal symmetric 2-form on $\E^0 \subset J^0 \pi$, $q_l$ as horizontal symmetric 2-forms on $\E^l \subset J^l \pi$, and $K$ as a horizontal vector field on $J^\infty \pi$. 

By adding the horizontal 
exterior derivative $\omega^{d+1}:=\hat{d}I_1^1$ of a first-order invariant, we obtain a complete horizontal coframe since $\omega^{d+1}(K)\neq 0$ generically. In other words, the $(d+1)$-form $\rho := \omega^1\wedge\cdots\wedge \omega^d\wedge\omega^{d+1}$ is nonzero on a Zariski open set
in $\E^k$ for every $k\ge2$.   
Note that $\omega^{d+1}$ has order two, while $\omega^i$ has order $1$ for $i \leq d$. We denote the dual invariant frame by $V_1, \dots, V_{d+1}$. These are invariant derivations of the form $V_i = \alpha_i^j D_j$ with the coefficients uniquely determined by the condition $\omega^i(V_j) = \delta^i_j$. 

Using the invariant frame $V_1,\dots, V_{d+1}$, we can construct the second-order differential invariants $g(V_i, V_j)$, $q_1(V_i,V_j)$, $q_2(V_i,V_j)$, along with the $k$-th order invariants $q_k(V_i,V_j)$. However, not all of these are independent. For example, $g(V_i, V_j)= 0$ when $i \neq j$ and $g(V_{d+1},V_{d+1})=0$, and similarly for $q_1$. Simultaneously, there are differential invariants that are not algebraic combinations of these invariants. This is easy to see from the Hilbert function, which is computed in Proposition \ref{prop:Hilbert} in Section \ref{S3.4}. 

In the particular case $d=2$ there exist 9 independent differential invariants of order 2, while $g(V_i,V_j)$, $q_1(V_i,V_j)$, $q_2(V_i, V_j)$ provide us with only $5$ independent invariants: 
 \[
q_1(V_1,V_1), \quad  q_1(V_2,V_2), \quad  q_2(V_1,V_1), \quad q_2(V_1,V_2),\quad q_2(V_2,V_2).
 \]
Let us discuss this case in more detail. It arguably encompasses the most important Carrollian spacetimes of $\dim M=3$, due to their occurrence on black hole horizons. 

 \begin{proposition} \label{prop:2Dlocal}
The following 9 differential invariants are independent as functions on $\E^2 \subset J^2 \pi$: 
 \[ 
I_1^1, \quad I_1^2, \quad q_2(V_1,V_1), \quad q_2(V_1,V_2), \quad q_2(V_2,V_2), \quad \hat{d}I_1^2 (V_1), \quad \hat{d}I_1^2 (V_2), \quad \hat{d}\omega^1(V_1,V_2), \quad \hat{d}\omega^2(V_1, V_2).
 \]
They provide a maximal set of functionally independent differential invariants of the second-order. 
 \end{proposition}

 \begin{proof}
The functional independence is by direct verification
via Maple: \href{https://doi.org/10.5281/zenodo.23317124}{doi.org/10.5281/zenodo.23317124}.  (Note that the remaining freedom in choosing eigenvectors $\omega^1,\omega^2$ does not affect the validity of the proposition.)

The fact that there are no more than 9 functionally independent invariants on $\E^2$ can be shown by analyzing the orbit dimension of the Lie pseudogroup of local diffeomorphisms. We will postpone this analysis to Section \ref{S3.4} and Proposition \ref{prop:Hilbert}.  
 \end{proof}

The invariant 1-forms $\omega^1, \dots, \omega^{d}$ are defined in terms of the eigenvalues of $Q_1^*$, which are solutions to an algebraic equation of degree $d$. Because of this, the 1-forms and the scalar invariants in Proposition \ref{prop:2Dlocal} that are constructed from them, are only locally invariant.
In the simplest case, the expressions involve radicals, for which a diffeomorphism may change branches, so there is no global invariance. 
In general, these algebraic equations do not possess a closed formula for solutions, which restricts practical computations even for these local branches. All these obstacles disappear with the approach of Theorem \ref{thm:GLT}.
In the next section, we will construct a set of 9 rational differential invariants. These are global, i.e., they are defined outside a Zariski-closed invariant set in $\E^\infty$.

\subsection{Rational canonical frame}\label{S3.3}

Let us identify the screen bundle $\U$ with $\ker \hat{d}I_1^1$. We impose the \textit{cyclic condition} that this $\U$ can be generated by powers of $Q_1$ from a single vector $\xi$.
Recall that a generic endomorphism has many cyclic vectors. The purpose of the following theorem is to \textit{invariantly} fix one such vector.

 \begin{theorem}[\textbf{Rational canonical frame}]
Consider the Lie pseudogroup action on the space of jets of Carrollian $(d+1)$ spacetimes $(M,g,K)$. 
We identify $\U\simeq\ker \hat{d} I_1^1$ and view $Q_1$ as an endomorphism of this $\U$. Consider a Zariski open proper subset, on which 
$\hat{d}I_1^1\wedge \hat{d}I_1^2\wedge (\hat{d}I_1^2 \circ Q_1)\wedge\dots 
\wedge \bigl(\hat{d}I_1^2\circ(Q_1)^{d-1}\bigr) \neq 0$. 
There exists a unique horizontal 
invariant vector field $\xi$ satisfying the conditions
 \begin{equation}\label{eq:canonical_frame}
\xi\in\ker\hat{d}I_1^1,\qquad 
\xi,\,Q_1\xi,\,\dots,\,(Q_1)^{d-2}\xi\in\ker\hat{d}I_1^2, \qquad  
\hat{d}I_1^2\bigl((Q_1)^{d-1}\xi\bigr)=1
 \end{equation}
and $(\xi, Q_1\xi, \dots, (Q_1)^{d-1}\xi, K)$ is a rational canonical invariant frame of order 2 on this set.
 \end{theorem}

\begin{proof}
The first $d$ conditions are equivalent to 
 \begin{equation}
\xi \in L := \ker \hat{d}I_1^1 \cap \ker \hat{d}I_1^2 \cap \ker \hat{d}I_1^2 
\circ Q_1 \cap \cdots \cap \ker \hat{d}I_1^2 \circ (Q_1)^{d-2} \subseteq \U. 
 \end{equation}
By linear independence of the indicated 1-forms, it follows that the common kernel $L$ is one-dimensional and additionally that 
$\bigl(\hat{d}I_1^2 \circ(Q_1)^{d-1}\bigr)|_L \neq 0$. 
Thus, the remaining condition $\hat{d}I_1^2\bigl((Q_1)^{d-1}\xi\bigr) = 1$ has a 
unique solution on this line $L$. Note that $\xi$ is cyclic for $Q_1$. 
(To see this, assume for the sake of obtaining a contradiction that there exists 
$1\leq k<d$ such that $(Q_1)^k\xi\in\langle\xi,Q_1\xi,\dots,(Q_1)^{k-1}\xi\rangle$. 
Applying $(Q_1)^{d-1-k}$ to $(Q_1)^k\xi$ yields an expression in the kernel 
of $\hat{d}I_1^2$ which is a contradiction with 
$\hat{d}I_1^2\bigl((Q_1)^{d-1}\xi\bigr)=1$). 
The solution $\xi$ is rational in jets because $I_1^2$ and $Q_1$ are rational 
together with \eqref{eq:canonical_frame} being linear in $\xi$. 
Canonicity and invariance readily follow from the same properties of $Q_1$.
\end{proof}

For the case $d = 2$, we now describe all second-order rational invariants. Let us denote 
$\eta = Q_1\xi$. In the $\xi, \eta$-basis, we have that 
\begin{equation}\label{eq:Q1_matrix}
    Q_1 = \begin{pmatrix}
   0 & \frac12 I_1^2 - \frac12 (I_1^1)^2 \\ 1 & I_1^1 
\end{pmatrix}
\end{equation}
 
First, we can contract $\xi, \eta$ with the Carrollian metric (viewed as a horizontal 
symmetric $2$-tensor on $J^{\infty}$) to produce invariants $g(\xi, \xi), g(\xi, \eta)$ 
and $g(\eta, \eta)$. Of course, there are only two new functionally independent 
invariants among these, as we have
\begin{equation}\label{eq:syzygy1}
    g(\eta, \eta) = g(Q_1(\xi), \eta) = g(\xi, Q_1(\eta)) = \bigl(\tfrac12 I_1^2 - \tfrac12 (I_1^1)^2\bigr)  g(\xi, \xi) + I_1^1 g(\xi, \eta).
\end{equation}

Moreover, we have that 
\begin{equation}
    Q_2(\eta) = J_1 \xi + J_2 \eta
\end{equation}
for two rational scalar differential invariants $J_1, J_2$. The trace of $Q_2$ is given 
by $I_2^1$ so the $\eta$-component of $Q_2(\xi)$ equals $I_2^1 - J_1$. An application of Cayley--Hamilton then yields the $\xi$-component, and we obtain the matrix representation 
\begin{equation}\label{eq:Q2_matrix}
    Q_2 = \begin{pmatrix}
J_1 & \frac{1}{2 J_2}\bigl(I_2^1(2 J_1 - I_2^1) - 2 J_1^2 + I_2^2\bigr) \\ J_2 & I_2^1 - J_1.    
\end{pmatrix}
\end{equation}

As a consequence, we get the following syzygy among the obtained invariants:
\begin{equation}\label{eq:syzygy2}
    g(Q_2\xi, \eta) = J_1\, g(\xi,\xi) + J_2\, g(\xi, \eta) = \left(\frac{1}{2 J_2}(I_2^1(2 J_1 - I_2^1) - 2 J_1^2 + I_2^2)\right)\,  g(\xi,\xi) + (I_2^1 - J_1)\, g(\xi, \eta).
\end{equation}

Recall from \cref{S3.2} that the vector fields $V_1, V_2$ are dual to $\omega^1, 
\omega^2$, and consider the two second-order invariants $
\hat{d}\omega^1(V_1, V_2)$ and $\hat{d}\omega^2(V_1, V_2)$. Two syzygies among the invariants were already exhibited in equations \eqref{eq:syzygy1} and \eqref{eq:syzygy2}. Now we show functional independence of nine rational invariants of order 2, which is the maximal possible number cf.\ Section \ref{S3.4}.

\begin{proposition} \label{prop:2Dglobal}
The following 9 second-order differential invariants are rational and algebraically independent  functions on the space of 2-jets of  Carrollian spacetimes:
\begin{equation}
    I_1^1,\ I_1^2,\ I_2^1,\ I_2^2,\ g(\xi,\xi),\ g(\xi, \eta),\ J_1,\ \hat{d}\omega^1(V_1, V_2)^2 + \hat{d}\omega^2(V_1,V_2)^2,\ \hat{d}\omega^1(V_1,V_2)^2 \,\hat{d}\omega^2(V_1,V_2)^2.
\end{equation}
\end{proposition}

\begin{proof}
Rationality of the first seven invariants is clear, since the 
operator $Q_1$ and the frame $(\xi, \eta, K)$ are rational in the jets of $(g,K)$. 

We prove the rationality of the last two invariants by showing that these are invariant under the group generated by the flipping of signs of the appearing roots. From 
\eqref{eq:Q1_matrix}, we obtain that the eigenvalues of $Q_1$ are given by $\lambda_{\pm} 
= \frac12 I_1^1 \pm \frac12 \sqrt{2 I_1^2 - (I_1^1)^2} $. In the notation of Section \ref{S3.2}, we take 
$\omega^1_0 := \lambda_{+}^{-1} dx + dy$ and $\omega^2_0 := \lambda_{-}^{-1} dx + dy$ as 
eigenvectors for $Q_1^{*}$. These eigenvectors are relative invariants, but can be made 
into absolute invariant eigenvectors by normalization, i.e., we define $\omega^1 := 
\frac{\omega^1_0}{|\omega^1_0|}$ and $\omega^2 := \frac{\omega^2_0}{|\omega^2_0|}$. 

Note that precisely three roots appear in the expressions of $\hat{d}\omega^i(V_1, V_2)$, 
namely $\sqrt{g(\omega^1_0,\omega^1_0)}, \sqrt{g(\omega^2_0, \omega^2_0)}$ and $\sqrt{2 I_1^2 - (I_1^1)^2}$. Now consider the following generators of our group:
\begin{equation}
    \sigma_1: \sqrt{g(\omega^1_0, \omega^1_0)} \rightarrow -\sqrt{g(\omega^1_0, \omega^1_0)},\qquad \sigma_2: \sqrt{g(\omega^2_0, \omega^2_0)} \mapsto -\sqrt{g(\omega^2_0, \omega^2_0)},\qquad \sigma_3: \lambda_{+} \mapsto \lambda_{-}\;.
\end{equation}
Note that $\sigma_1$ fixes $V_2$ but assigns $V_1 \mapsto -V_1$ and $\omega^1 \mapsto -
\omega^1$. Similarly, $\sigma_2$ fixes $V_1$ but assigns $V_2 \mapsto -V_2$ and $\omega^2 
\mapsto -\omega^2$. Finally, we see that $\sigma_3$ interchanges $V_1 \leftrightarrow 
V_2$ and $\omega^1 \leftrightarrow \omega^2$. Actually, $\sigma_3$ also interchanges $
\sqrt{g(\omega^1_0, \omega^1_0)} \leftrightarrow \sqrt{g(\omega^2_0, \omega^2_0)}$. The corresponding
monodromy group $\Gamma$ (that interchanges branches) is not Abelian, it is isomorphic to the dihedral group $D_4$ of order 8.

One readily computes the action of each $\sigma_i$ on $\hat{d}\omega^j(V_1, V_2)$ for $j=1,2$, for instance 
$\sigma_3(\hat{d}\omega^1(V_1, V_2)) = \hat{d}\omega^2(V_2, V_1) = -\hat{d}\omega^2(V_1, V_2)$. 
In view of these actions, we conclude that 
$\hat{d}\omega^1(V_1, V_2)^2 + \hat{d}\omega^2(V_1,V_2)^2$ and $
\hat{d}\omega^1(V_1,V_2)^2 \,\hat{d}\omega^2(V_1,V_2)^2$ are invariant under $\Gamma$, and 
are thus rational in the jets of $(g,K)$.

Independence of the invariants 
was checked in Maple: \href{https://doi.org/10.5281/zenodo.23317124}{doi.org/10.5281/zenodo.23317124}.
\end{proof}

In the introduction, we showed one way of generating all invariants from $d+1$ basis differential invariants $I^0, \dots, I^{d+1}$. In the jet framework, we can rephrase it like this: on the domain where 
$\hat dI^0\wedge\cdots\wedge\hat{d}I^{d+1}\ne0$  
construct the horizontal coframe $\hat dI^i$ and the dual horizontal frame of Tresse derivatives $\hat \partial_{I^j}=\alpha_j^l D_l$, with coefficients defined by the condition $\hat dI^i(\hat \partial_{I^j})=\delta_j^i$. The functions $G_{ij}=g(\hat \partial_{I^i}, \hat \partial_{I^j})$ and $H^i = \hat dI^i(K)$ are now differential invariants. Since these differential invariants fully encode the Carrollian spacetime (locally), they serve as generators for the differential invariants, thus realizing Theorem \ref{thm:GLT}. Notice that even though they, in principle, generate all invariants, it is not clear how to obtain, for example, $\hat d\omega^1(V_1,V_2)$ from Proposition \ref{prop:2Dlocal} if the 3 basis differential invariants are $I_1^1, I_1^2, I_2^1$. 

The generating set $\{I^i, G_{ij}, H^i, \hat \partial_{I^i}\}$ is not a minimal set of generators. 

 \begin{proposition}
The 6 rational  differential invariants 
 \begin{equation}\label{eq:6generators}
I_1^1,\quad I_1^2,\quad I_2^1,\quad I_2^2,\quad \hat{d}\omega^1(V_1, V_2)^2 + \hat{d}\omega^2(V_1,V_2)^2,\quad \hat{d}\omega^1(V_1,V_2)^2 \,\hat{d}\omega^2(V_1,V_2)^2.
 \end{equation} 
and the 3 invariant derivations $\xi, \eta, K$ generate 
a transcendence basis for the field of differential invariants (of any order). 
 \end{proposition}

Let us recall that a transcendence basis is a maximal independent subset of the field of invariants, which implies
that any other differential invariant can be obtained
by an algebraic extension (simplest case: extracting a root of a power). 
Actually, this extension is only required for invariants of order 2 because the higher order invariants obtained by invariant derivations are affine in their top-order jets.

 \begin{proof}
For $d=2$, the invariants \eqref{eq:6generators} and the three invariant derivations generate 9 independent invariants of order 2, in addition to 16 new independent invariants of order 3. This was verified with the help of Maple: \href{https://doi.org/10.5281/zenodo.23317124}{doi.org/10.5281/zenodo.23317124}. We thus obtain a transcendence basis of the field of rational differential invariants of order 3 (compare to the counting in Table \ref{tab:Hilbert}).  

Looking back at the previous generators, let us choose basic invariants $I^0, I^1, I^2$ of  second order; the respective Tresse derivations $\hat\p_{I^i}$ are of  third order. This implies that the corresponding functions $G_{ij}$ and $H^i$ are third-order invariants. Thus, we can generate all higher order differential invariants from those of order $3$, and the statement of the proposition follows.    \end{proof}

 \begin{remark}
As is well-known, cf.\ \cite{S}, all invariants of a frame
$V_i$ with dual coframe $\omega^j$ on a (finite-dimensional) manifold can be expressed through structure functions 
$c_{ij}^k=\omega^k([V_i,V_j])=-d\omega^k(V_i,V_j)$ and their derivatives along $V_i$. 
The frames we construct restrict to jet-sections,
that are Carrollian structures, and they encode those structures 
(under a minor modification, by normalizing $\omega^{d+1}$ or equivalently changing $V_{d+1}$ to $K$).
While in the framework of the previous subsection all $c_{ij}^k$ have order 2,
the structure functions in this section have orders 2 and 3.
Thus, generating second-order differential invariants from them would involve syzygies among the 
structure functions and their invariant derivatives.
 \end{remark}

\subsection{Counting invariants: The Hilbert and Poincaré functions}\label{S3.4}

We stated in Section \ref{S3.2} that there are 9 independent differential invariants of second order when $d=2$. In this section, we justify this and count the number of independent invariants of any order for $d\geq2$. 

Let us introduce local coordinates $x^1, \dots, x^d, t$ on $M$
in which the vector field $K$ is straightened:
 \begin{equation}\label{straightened}
K = \partial_t, \quad  g = g_{ij}(x,t)\,dx^i\,dx^j \quad (g_{ji} = g_{ij}). 
 \end{equation}
The vector fields on $M$ preserving this class 
(in particular, commuting with $K$)
of  Carrollian structures take the general form 
 \[
X= a^i(x) \partial_{x^i} + b(x) \partial_t,
 \]
For what follows, we note that the local straightening of $K$
does not influence the count of differential invariants.
 
The Carrollian structures given by \eqref{straightened}
with $\op{rank}(g_{ij})=d$
are sections of a subbundle $E\subset S^2T^*M$ defined by the condition $g(\p_t,\cdot)=0$, and we keep the same notation $u_{ij}=u_{ji}$ for the coordinates in the fiber (only now with $i,j \leq d$),
and $\pi$ for the projection. 
Then the lift of $X$ to $E$ takes the form
 \begin{equation} \label{eq:Xhat}
\hat X= a^i \partial_{x^i} + b \partial_t-\sum_{i\leq j}\left(a^l_{x^i} u_{lj} + a^l_{x^j} u_{il}\right) \partial_{u_{ij}}. 
 \end{equation}
Repeated indices indicate a sum from $1$ to $d$ unless otherwise specified. 
Notice that $\hat X$ depends on the 1-jet of functions $a^i$ and the 0-jet of $b$. Similarly, as we will see below, the general prolonged vector field $\hat X^{(k)}$ depends on the $(k+1)$-jet of $a^i$ and the $k$-jet of $b$. Thus, the dimension of the span of these vector fields in $J_{\theta_k}^k(E)$ is bounded from above by 
 \[
\dim J_p^{k+1}(M,\mathbb R^{d}) + \dim J_p^k(M,\mathbb R) = d\,\binom{d+k+1}{d}+\binom{d+k}{d},
 \]
where $J_p^k(M, \R^i)$ denotes the space of $k$-jets of maps from $M$ to $\R^i$ at the point $p\in M$ (the above equality does not depend on $p$). 
Since $\dim J^k(E) = d+1+\binom{d+1}{2}\binom{d+k+1}{d+1}$, we obtain a lower bound for the number of independent differential invariants of order $k$: 
 \begin{equation} \label{eq:Hilbert0}
r_{d,k} =d+1+\binom{d+1}{2}\binom{d+k+1}{d+1}-d\, \binom{d+k+1}{d}-\binom{d+k}{d}.
 \end{equation} 

 \begin{definition}
We let $s_{d,k}$ denote the number of algebraically independent differential invariants of order $k$, and define the Hilbert function $h_{d,k}$ by $h_{d,0} = s_{d,0}$ and $h_{d,k}=s_{d,k} - s_{d,k-1}$, $k >0$.
 \end{definition}

Thus, $s_{d,k}$ counts the total number of algebraically independent differential invariants of order $k$, while $h_{d,k}$ counts the number of independent differential invariants of ``pure order'' $k$. In proposition \ref{prop:Hilbert}, we will see that the lower bound \eqref{eq:Hilbert0} is usually attained for $k \geq 3$, i.e., $s_{d,k} = r_{d,k}$. The only exception is $d=2$, in which case $s_{2,k} = r_{2,k} +2$ for $k \geq 1$. 

 \begin{proposition}\label{prop:Hilbert} 
The Hilbert function for differential invariants of  Carrollian structures ($d \geq 2$) 
is given by $h_{d,0}=0$, $h_{d,1}=d$, then 
$h_{d,2} = r_{d,2} -d+2 \delta_{d,2}$ for $k=2$ and, in general, 
 \[
h_{d,k} = r_{d,k}-r_{d,k-1}= \binom{d+1}{2}\binom{d+k}{d}-
d\,\binom{d+k}{d-1}-\binom{d+k-1}{d-1}\quad\text{for }\ k \geq 3.
 \] 
 \end{proposition}

 \begin{proof}
The action is transitive on $M$, so we fix a point $0\in M$ and consider the vector fields of the form \eqref{eq:Xhat} vanishing at $0$. This stabilizer (isotropy) algebra of 0 is 
$\U^*\otimes\U$, parametrized by $d^2$ constants $a_{x^i}^l$.
Let us split the tangent space 
$T_0M = \U\oplus\K$, where $\U=\langle \p_{x^1},\dots, \p_{x^{d}}\rangle$ and $\K = \langle\p_t\rangle$. 
(This splitting is not canonical but useful in computations.) 
By \eqref{eq:Xhat}, the action of the stabilizer $\U^*\otimes\U$ on $E_0\simeq S^2\U^*$ is standard and hence has an open orbit (Sylvester's law of inertia), implying $h_{d,0}=0$. The stabilizer of a generic point $\theta_0$ in this space
(given by $\op{rank}(u_{ij})=d$) is $\mathfrak{st}_0\simeq \mathfrak{so}(\U)\simeq\Lambda^2\U^*$. 

We consider the vector fields whose lift vanishes on $\theta_0$ and analyze how they act on the fiber of $J^1 E \to J^0 E$ and higher order jet bundles. It is a general fact (see \cite[§1.1]{KL1}) that $\pi_{k+1,k}\colon J^{k+1} E \to J^{k} E$ is an affine bundle with the fiber $\pi_{k+1,k}^{-1}(\theta_{k})\simeq S^kT_{\pi_k(\theta_k)}^*M\otimes E_{\pi_k(\theta_k)}$, $\theta_k \in J^k E$. We will use this fact extensively below, together with the decomposition $T_0M = \U \oplus \K$. To simplify the formulas, we will 
use jet-notations $a^k_i=a^k_{x^i}$, $a^k_{ij}=a^k_{x^ix^j}$, and similar for $b$, while $u_{ij,k}=(u_{ij})_{x^k}$, $u_{ij,t}=(u_{ij})_{t}$, $u_{ij,kt}=(u_{ij})_{x^kt}$, etc. 
If $j < i$, we will naturally identify $a^k_{ji}$ with $a^k_{ij}$. 

Now, fix the point $\theta_0 \in J^0 E$ by $t=0$, $x^i=0$, $u_{ij}=\delta_{ij}$, which lies in the open orbit of $E$.  By the standard formula for jet-prolongation (see 
\cite[formula (2)]{KL1} or \cite[Theorem 4.16]{O}), 
the vector fields that vanish on $\theta_0$ satisfy
 \begin{equation} \label{eq:Xhat1}
\hat X^{(1)}|_{\pi_{1,0}^{-1}(\theta_0)} = -\sum_{i\leq j} \Big( a^j_{il}  + a^i_{jl} +   b_{l} u_{ij,t} +\gamma_{ijl} \Big) \partial_{u_{ij,l}} -  \sum_{i \leq j}(a_i^m u_{mj,t} + a_j^m u_{im,t}) \partial_{u_{ij,t}},
 \end{equation} 
where $a^j_{i} = -a^i_{j}$ and 
$\gamma_{ijl}= a_i^m u_{jm,l} + a_j^m u_{im,l} +  a_l^m u_{ij,m}$. 
The action of the vector fields vanishing on $\theta_0$ corresponds to an action of $(S^2\mathbf{U}^* \otimes \mathbf{U}) \oplus (\mathbf{U}^* \otimes \mathbf{K}) \oplus \mathfrak{st}_0$ on 
 \[ 
T_0^*M \otimes E_0 \simeq (\mathbf{U}^* \otimes S^2 \mathbf{U}^*) \oplus (\mathbf{K}^* \otimes S^2 \mathbf{U}^*).
 \]
From the second term of \eqref{eq:Xhat1}, we see that the previous stabilizer resolves at this step, as orbits of $\mathfrak{st}_0 \simeq \Lambda^2 \mathbf{U}^*$ in $\mathbf{K}^* \otimes S^2 \mathbf{U}^*$ are $\binom{d}{2}$-dimensional. Next, keeping $a_i^m$ fixed in the first term of \eqref{eq:Xhat1}, we see that $(S^2\mathbf{U}^* \otimes \mathbf{U}) \oplus (\mathbf{U}^* \otimes \mathbf{K})$ acts transitively on $\mathbf{U}^* \otimes S^2 \mathbf{U}^*$, and we have $\mathfrak{st}_1 \simeq \mathbf{U}^* \otimes \mathbf{K}^*$, as the first term vanishes if and only if
\begin{equation}\label{eq:sola2}
    2a^j_{il} = -u_{ij,t} b_l + u_{li,t} b_j- u_{jl,t} b_i - \gamma_{ijl} + \gamma_{lij} - \gamma_{jli}.
\end{equation} 
Thus, we can compute the number of independent differential invariants of first order simply by comparing dimensions:  \[h_{d,1} = \dim S^2 \mathbf{U}^* - \dim \Lambda^2 \mathbf{U}^*=d.\] 

It is clear from \eqref{eq:Xhat1} that we can use the action to fix the point $\theta_1$ by $u_{ij,l}=0$,  $u_{ij,t} = \lambda_i \delta_{ij}$, and $\pi_{1,0}(\theta_1)=\theta_0$. The parameters $\lambda_i$ are the eigenvalues of the operator $Q_1$ at the point $\theta_1$, and we will consider a point in general position where they are all different. The vector fields that vanish on $\theta_1$ satisfy 
 \begin{align*}
\hat X^{(2)}|_{\pi_{2,1}^{-1}(\theta_1)} = &- \sum_{i \leq j} \sum_{l \leq m}\Big(a_{ilm}^j + a_{jlm}^i+\lambda_j \delta_{ij} b_{lm}+b_l u_{ij,mt} + b_m u_{ij,lt}\Big) \p_{u_{ij,lm}} \\
&+ \sum_{i \leq j} \left(\lambda_j^2 \delta_{ij} b_{l} +\frac{1}{2} (\lambda_i-\lambda_j) 
\big(\lambda_ib_j\delta_{il}- \lambda_jb_i\delta_{jl}\big) -b_lu_{ij,tt}\right)\p_{u_{ij,l t}},
 \end{align*} 
where we have used \eqref{eq:sola2}. This action can be identified with the action of $(S^3 \mathbf{U}^* \otimes \mathbf{U}) \oplus (S^2 \mathbf{U}^* \otimes \mathbf{K}) \oplus \mathfrak{st}_1$ on 
\[ S^2 T_0^*M \otimes E_0 \simeq (S^2 \mathbf{U}^* \otimes S^2 \mathbf{U}^* )\oplus (\mathbf{K}^* \otimes \mathbf{U}^* \otimes S^2 \mathbf{U}^*) \oplus( S^2 \mathbf{K}^* \oplus S^2 \mathbf{U}^*). \]
There is no action on the third component, while $\mathfrak{st}_1 \simeq \mathbf{U}^* \otimes \mathbf{K}^*$ acts on the second component. Since the action of $\mathfrak{st}_1$ on $\K^* \otimes \U^* \otimes S^2 \U^*$ has $d$-dimensional orbits, the stabilizer $\mathfrak{st}_1$ resolves. The rest acts on the first component. For $d \geq 3$, this action is free. Notice that this is possible since 
\[ \dim \left( (S^3 \mathbf{U}^* \otimes \mathbf{U}) \oplus (S^2 \mathbf{U}^* \otimes \mathbf{K})\right) \leq  \dim \left(S^2 \mathbf{U}^* \otimes S^2 \mathbf{U}^*\right). \]
This gives the expected count $h_{d,2} = r_{d,2}-d$ for $d \geq 3$. However, for $d=2$ we have
\[ \dim \left( (S^3 \mathbf{U}^* \otimes \mathbf{U}) \oplus (S^2 \mathbf{U}^* \otimes \mathbf{K})\right) =11, \quad  \dim \left(S^2 \mathbf{U}^* \otimes S^2 \mathbf{U}^*\right) = 9,  \]
and we get a 2-dimensional stabilizer $\mathfrak{st}_2$, and $h_{3,2} = (r_{3,2}-2)+2 = 7$.

This pattern continues for jets of order $k >2$. We have 
\begin{equation} \label{eq:SkT}
S^k T_0^* \otimes E_0 \simeq (S^k \mathbf{U}^* \otimes S^2 \mathbf{U}^* )\oplus (\mathbf{K}^* \otimes S^{k-1} \mathbf{U}^* \otimes S^2 \mathbf{U}^*) \oplus \cdots \oplus (S^{k}\mathbf{K}^*\otimes S^2 \mathbf{U}^*).
\end{equation}
and the vector fields vanishing at a generic point $\theta_{k-1}$ (satisfying $\pi_{k-1}(\theta_{k-1})=0$) correspond to 
\[(S^{k+1} \mathbf{U}^* \otimes \mathbf{U}) \oplus (S^k \mathbf{U}^* \otimes \mathbf{K}) \oplus \mathfrak{st}_{k-1},  \]
acting on $S^k \U^* \otimes \U$. For $d>2$ and $k \geq 2$, $\mathfrak{st}_k=0$, and the action is free. The formula in the proposition follows from \[ h_{d,k} = \dim(S^{k} \mathbf{U}^* \otimes S^2 \mathbf{U}^*) - \dim\left((S^{k+1} \mathbf{U}^* \otimes \mathbf{U})\oplus (S^k \mathbf{U}^* \otimes \mathbf{K}) \right).\]
For $d = 2$, we have $\dim \mathfrak{st}_k = 2$ for $k \geq 2$. The situation is similar to the case $k=2$ that we looked at above since
\[ \dim(S^{k} \mathbf{U}^* \otimes S^2 \mathbf{U}^*)  = 3k+3, \qquad \dim\left((S^{k+1} \mathbf{U}^* \otimes \mathbf{U})\oplus (S^k \mathbf{U}^* \otimes \mathbf{K}) \right) = 3k+5. \]
Therefore, the action never becomes free. However, $\mathfrak{st}_{k-1}$ always resolves when increasing the order by $1$, as it acts freely on the second component of \eqref{eq:SkT} for $k \geq 2$. Thus, the same formula holds for $h_{2,k}$, the only exception being $h_{2,2}$, where the 2-dimensional stabilizer first appears.  
 \end{proof}

\begin{table}[]
    \centering
\begin{tabular}{|| c | c c c c c||} 
 \hline
 $d \backslash k$ & 1 & 2 & 3 & 4 & 5 \\ [0.5ex] 
 \hline\hline
 2 & 2 & 7 & 16 & 28 & 43  \\ 
 \hline
 3  & 3 & 21 & 65 & 132 & 231   \\ 
 \hline
  4 & 4 & 56 & 190 & 441 & 868  \\ [1ex]  
  \hline
\end{tabular}
    \caption{Values of the Hilbert function $h_{d,k}$ for small values of $d$ and $k$. }
    \label{tab:Hilbert}
\end{table}

\begin{remark}
    Notice that $h_{2,0}+h_{2,1}+h_{2,2} = 9$, which completes the proof of Proposition \ref{prop:2Dlocal}.
\end{remark}

The Hilbert function can be compactly encoded in terms of the corresponding Poincaré function (a rational function for the moduli \cite{BK}), defined by 
 \[ 
p_d(z) = \sum_{i=0}^\infty h_{d,k} z^k.
 \]
The computation of this from the formula for $h_{d,k}$ from Proposition \ref{prop:Hilbert} is elementary when taking advantage of the binomial formula 
\[  \frac{1}{(1-z)^{d}} = \sum_{k=0}^\infty \tbinom{d+k-1}{k} z^k. \]

\begin{corollary}
    The Poincaré function for differential invariants of  Carrollian structures is given by 
    \begin{align*}
        p_2(z) = \frac{z(2+z+z^2-z^3)}{(1-z)^3}
    \end{align*}
    when $d=2$, while for $d > 2$ it is given by 
    \[p_d(z)=\frac{2z^2 + (d^2 + 3d - 2)z - 2d}{2z(1-z)^{d+1}}+\frac{d}{z}-d \cdot  z^2-\frac{d(d-3)}{2} z +\frac{d(d-1)}{2}+1.\]
\end{corollary}

\section{Fundamental invariants and Symmetry}\label{S4}

We begin with a computation of the Spencer cohomology, which contains
algebraic information on invariants and symmetry of the structure.
We refer to \cite{S,KL1,BK} for details on this formalism, its relation to
the jet framework, and implications for geometric structures.

\subsection{Spencer Cohomology of Carrollian spacetimes}

The Lie equation encodes symmetries of geometric structure of the given type.
Formally, it starts with a symbol of the structure and linear stabilizer. 
For the Carrollian structure, the group $G\subset\op{End}(\T)$
of linear transformations of the tangent space $\T=T_oM$ at a point $o\in M$  preserving $(g,K)$ is $SO(\U)\ltimes \U^*$. Its Lie algebra
$\g=\op{Lie}(G)$ is given by the exact sequence
 $$
0\to \K\otimes\U^*\longrightarrow\g\longrightarrow\Lambda^2\U^*\to0,
 $$
where we identify $\mathfrak{so}(\U)$ with skew-symmetric matrices via $g$; 
we also identify $\U\simeq\U^*$ in what follows.

For the flat model of the Carrollian spacetime 
 \begin{equation}\label{flat}
M=\R^{d+1}(x^1,\dots,x^d,t),\quad K=\p_t,\quad g=\sum_{i=1}^d(dx^i)^2
 \end{equation}
the Lie group $G$ is generated by rotations in $\R^d(x)$ space and
transformations $(x,t)\mapsto(x,t+a\cdot x)$.

Note the canonical exact sequence 
 $$
0\to\K\longrightarrow\T\longrightarrow\U\to0,
 $$
whence $\T\simeq\K\oplus\U$. For this splitting one notes that 
$\K$ is the injective part, while $\U$ is the projective part. 
Dually, $\U^*$ is the injective part of $\T^*\simeq\K^*\oplus\U^*$, 
while $\K^*$ is the projective part.
We will exploit such non-canonical splitting in what follows 
(they arise from gradings associated with filtration for spectral sequences used to compute cohomology, but for the sake of simplicity, we omit this abstract technique).

The Cartan--Sternberg prolongation of $\g$ gives the symbolic system 
$\g_k=\g^{(k-1)}:=(\g\otimes S^{k-1}\T^*)\cap(T\otimes S^k\T^*)$. 
In our case, since the prolongation of the orthogonal algebra is trivial, 
we get:
 $$
\g_k=\K\otimes S^k\U^*\subset \T\otimes S^k\T^*,\quad k>1,
 $$
and we complete this by $\g_0=\T$, $\g_1=\g$. 
Note that $\g_0\oplus\g_1$ is a trivial one-dimensional extension
of what Lévy-Leblond named the {\it Carroll algebra} \cite{Duval,Figueroa},
while $\oplus_{k=0}^\infty\g_k$ is an infinite-dimensional Lie algebra
(subalgebra in the algebra of formal vector fields).
This algebra is closely related to BMS algebra, see \cite{Duval2}.

The Spencer complex, associated to this symbolic system $\{\g_k\}$, is 
 $$
\dots\to\g_{i+1}\otimes\Lambda^{j-1}\T^*\stackrel{\delta}\to
\g_i\otimes\Lambda^j\T^*\stackrel{\delta}\to
\g_{i-1}\otimes\Lambda^{j+1}\T^*\to\dots
 $$ 
where $\delta$ is the symbol of the de Rham operator.
Its cohomology is denoted by $H^{i,j}(\g)$.

 \begin{theorem}\label{Spencer1}
The non-trivial Spencer cohomology groups for a Carrollian spacetime are the following:
 \begin{gather*}
H^{0,0}=\T,\quad H^{0,1}=\T\oplus S^2\U^*,\quad H^{0,2}=S^2\U^*,\\
H^{1,2}=\Gamma(2\pi_2),\quad 
H^{1,k}=\Gamma(\pi_2+\pi_{k-1})\oplus\Gamma(\pi_2+\pi_k)\ \text{ for }\ 2<k<d,\\ 
H^{1,d}=\Gamma(\pi_2+\pi_{d-1})\oplus\Gamma(\pi_2),\quad H^{1,d+1}=\Gamma(\pi_2),
 \end{gather*}
where $\pi_i$ is the $i$-th fundamental representation of 
$A_{d-1}=\mathfrak{sl}(U)$ (note projective invariance
of the curvarture modules) and $\Gamma(\varpi)$ is the representation of weight $\varpi=\sum m_i\pi_i$. 
 \end{theorem}

In particular, the full curvature $H^{\bullet,2}(\g)$ consists of 
the intrinsic torsion $H^{0,2}$ and the intrinsic curvature $H^{1,2}$; 
the system $\g$ is involutive after one prolongation ($H^{i,\bullet}=0$ for $i>1$)
similar to Riemannian geometry.

 \begin{proof}
The 0-th Spencer complex consists of one term $\T$, whence $H^{0,0}$.
In the first complex
 $$
0\to(\K\otimes\U^*)\oplus(\Lambda^2\U^*)\longrightarrow\T\otimes\T^*\to0
 $$
we decompose the second nontrivial term into summands
$(\K\otimes\K^*)\oplus(\K\otimes\U^*)\oplus
(\U\otimes\K^*)\oplus(\U\otimes\U^*)$, whence the summands
$\T\simeq (\K\otimes\K^*)\oplus(\U\otimes\K^*)$ and 
$S^2\U^*\simeq(\U\otimes\U^*)/(\Lambda^2\U^*)$.
These encode the Lie equations for $\g$: preservation of $K$ and of $g$.

The second Spencer complex is well-known, and its computation gives
the intrinsic torsion; we refer to \cite{Figueroa,Figueroa1} for details.
The next Spencer complex is
 $$
0\to\K\otimes S^3\U^*\longrightarrow
\K\otimes S^2\U^*\otimes\T^*\longrightarrow
((\K\otimes\U^*)\oplus(\Lambda^2\U^*))\otimes\Lambda^2\T^*
\longrightarrow\T\otimes\Lambda^3\T^*\to0.
 $$
Note that it contains the exact ``Koszul" subcomplex 
$(\K\otimes S^{3-j}\U^*\otimes\Lambda^j\U^*,\delta)$; quotient by it
contains another exact ``Koszul" subcomplex 
$(\K\otimes S^{2-j}\U^*\otimes\Lambda^j\U^*\otimes\K^*,\delta)$.
Since this quotient does not change the cohomology, it suffices to consider
the resulting complex
 $$
0\to(\Lambda^2\U^*\otimes\K^*\otimes\U^*)\oplus
(\Lambda^2\U^*\otimes\Lambda^2\U^*)\longrightarrow
(\U\otimes\K^*\otimes\Lambda^2\U^*)\oplus
(\U^*\otimes\Lambda^3\U^*).
 $$
This, in turn, contains an exact two-term subcomplex, removing which we get
the only nonzero cohomology $H^{1,2}=\op{Ker}(\beta:\Lambda^2\U^*\otimes\Lambda^2\U^*\to\U^*\otimes\Lambda^3\U^*)$.
The map $\beta$ (restriction of $\delta$) is skew-symmetrization,
so its kernel is the space of 4-tensors satisfying the first Bianchi's
identity, and this allows to identify $H^{1,2}$ with the space
of algebraic curvature tensors $\mathcal{R}=\op{Ker}(S^2\Lambda^2\U^*\to\Lambda^4\U^*)$.

The higher Spencer cohomology is computed similarly: the $(k+1)$-st
complex contains the exact ``Koszul" subcomplex 
$(\K\otimes S^{k+1-j}\U^*\otimes\Lambda^j\U^*,\delta)$; quotient by it
contains another exact ``Koszul" subcomplex 
$(\K\otimes S^{k-j}\U^*\otimes\Lambda^j\U^*\otimes\K^*,\delta)$.
The resulting two-terms complex is filtered with the corresponding graded pieces
 \begin{equation}\label{quot1}
0\to \Lambda^2\U^*\otimes\Lambda^{k-1}\U^*\otimes\K^*\to \U^*\otimes\Lambda^k\U^*\otimes\K^*\to0
 \end{equation}
and 
 \begin{equation}\label{quot2}
0\to \Lambda^2\U^*\otimes\Lambda^k\U^*\to \U^*\otimes\Lambda^{k+1}\U^*\to0.
 \end{equation}
The only nontrivial differentials are surjective, so the only nontrivial cohomology 
is $H^{1,k}$ split into two pieces. To compute those for \eqref{quot2}
note the $A_{d-1}$ decomposition into irreps by the Littlewood–Richardson rule:
 $$
\Gamma(\pi_2\otimes\pi_m)=\Gamma(\pi_{m+2})\oplus\Gamma(\pi_1+\pi_{m+1})\oplus\Gamma(\pi_2+\pi_m)
=\Gamma(\pi_1\otimes\pi_{m+1})\oplus\Gamma(\pi_2+\pi_m),
 $$
in which we identify $\Gamma(\pi_d)$ as the trivial module and $\Gamma(\pi_m)=0$ for $m>d$.
This gives the cohomology $\Gamma(\pi_2+\pi_k)$ for \eqref{quot2},
and similarly $\Gamma(\pi_2+\pi_{k-1})$ for \eqref{quot1}.
 \end{proof}

\subsection{Intrinsic torsion and intrinsic curvature}

The spaces $H^{*,2}(\g)$ are important as a home of obstructions to flatness,
measured as compatibility for the Lie equation on symmetries, but also
as the space of fundamental invariants. Leaving discussion of symmetry till the next section, we concentrate on fundamental invariants. 
For Carrollian spacetimes $H^{*,2}=H^{0,2}\oplus H^{1,2}$ is split
as the space of torsions and the space of curvatures and we discuss them in turn. The story of the intrinsic torsion is well-known, cf.\ \cite{Figueroa,Figueroa1}, but we present it here in our language for completeness; the part about
the intrinsic curvature is apparently novel.

Consider the space $\mathcal{T}$ of {\em compatible} connections on $M$, given by 
the conditions $\nabla K=0$, $\nabla g=0$. Since pointwise the Carrollian structure 
has no moduli, such connections always exist.
Fixing one of them $\nabla\in\mathcal{T}$, any other is obtained by
the gauge transformation $\tilde{\nabla}=\nabla+A$, where 
$A\in\T\otimes\T^*\otimes\T^*$, $\T\otimes\T\ni(X,Y)\mapsto A_XY\in\T$,
is a tensor satisfying the following conditions to ensure 
conservation of compatibility:
 $$
A_XK=0,\quad g(A_XY,Z)+g(Y,A_XZ)=0\quad\forall X,Y,Z\in\T. 
 $$
Denote the space of such tensors by $\mathcal{A}$. 
The intrinsic torsion is an invariant part of torsion $T_{\tilde\nabla}$, when $A=\tilde\nabla-\nabla$ variates over $\mathcal{A}$. Note that $T_\nabla(\cdot,K)$ is an endomorphism of $\T$
and it can be naturally split into a symmetric and skew-symmetric parts. We claim that the former is independent of $\nabla\in\mathcal{T}$, while the residual gauge group acts transitively on the latter.

 \begin{proposition}
The intrinsic torsion is $T_\nabla(\cdot,K)_{\op{sym}}=-\tfrac12Q_1\in S^2\U^*=H^{0,2}(\g)$.
 \end{proposition}

 \begin{proof}
Let us split $\T=\K\oplus\U$ and decompose $A=A^K+A^U$ accordingly
for the values of $A_XY$. Then $A^K_X$ is an arbitrary 1-form on $\U$,
while $A^U_X$ is in $\mathfrak{so}(\U)$ for any $X\in\T$. 
For $X\in\U$ we have:
 $$
T_{\tilde\nabla}(X,K)= T_\nabla(X,K)-A^K_KX-A^U_KX.
 $$
The term $A^K_KX$ is an arbitrary vector in $\K$, so this component can be 
eliminated. The operator $A^U_K$ is arbitrary orthogonal, which eliminates the
skew-symmetric part in $\U$-value of this part of torsion, so that we get
$T_\nabla(\cdot,K)\in\op{End}_{\op{sym}}(\U)$.
For $X,Y\in\U$ we have:
 $$
T_{\tilde\nabla}(X,Y)= T_\nabla(X,Y)+(A^K_XY-A^K_YX)+(A^U_XY-A^U_YX).
 $$
The first parenthetical expression realizes arbitrary 2-form on $\U$,
while the second parenthetical expression realizes arbitrary tensor
in $\U\otimes\Lambda^2\U^*$. Thus we can completely eliminate this part of torsion:
$T_\nabla|_\U=0$. 

The remaining gauge is subject to conditions $A_K=0$ and $A_XY=A_YX$ for $X,Y\in\U$,
which describes the first prolongation $\g^{(1)}=\K\otimes S^2\U^*$. Finally, 
for $X,Y\in\U$ we compute
 \begin{multline*}
q_1(X,Y)=(L_Kg)(X,Y)=K\cdot g(X,Y)-g(L_KX,Y)-g(X,L_KY)\\
=g(\nabla_KX-L_KX,Y)-g(X,\nabla_KY-L_KY)=g(T_\nabla(K,X),Y)+g(X,T_\nabla(K,Y)),
 \end{multline*}
which identifies the invariant part of torsion with what we call
the intrinsic torsion in Section \ref{S2}. 
 \end{proof}

Given a choice of subbundle $\U\subset TM$ (splitting of the exact sequence for the
screen bundle) and the restrictions in the normalization of the proposition, let us
call a compatible connection $\nabla\in\mathcal{T}$ {\em minimal} if
$T_\nabla|_\U=0$,
$\op{Im}\bigl(T_\nabla(\cdot,K)\bigr)\subset\U$ and
$T_\nabla(\cdot,K)\in\op{End}(\U)$ is $g$-symmetric.

Now let us discuss the intrinsic curvature. In general, the next obstruction
from $H^{1,2}(\g)$ is not well-defined unless the previous $Q_1\in H^{0,2}(\g)$
vanishes. Actually, if $Q_1=0$ then $g$ is projectable along $\K$ to the 
(local or global) quotient-manifold $U$ and then the fundamental invariant is 
the curvature $R_g\in\mathcal{R}$ on $U$. Note that the flatness of the Carrollian structure is equivalent to the vanishing of the intrinsic torsion $Q_1$ and 
the intrinsic curvature $R_g$, which is the reason to call them fundamental invariants.

However for generic Carrollian spacetimes we still have well-defined intrinsic curvature.

 \begin{proposition}
If $I_2^1=K\cdot I_1^1\neq0$, denote $\U=\op{Ker}(dI_1^1)$ a canonical choice of 
the screen subbundle and by $\pi:\T\to\U$ the canonical projection along $\K$.
Restrict the connection $\nabla\in\mathcal{T}$ to be $\U$-minimal. 
Then the tensor $\pi R_\nabla|_\U\in\mathcal{R}=H^{1,2}(\g)$ is well-defined.
 \end{proposition}

 \begin{proof}
Variations in the class of connections of the proposition are given by 
$\tilde\nabla=\nabla+A$, where the tensor $A\in\K\otimes S^2\U^*$. Thus
$A_X\in \K\otimes\U^*$ implying $[A_X,A_Y]=0$ for $X,Y\in\U$, and we also have:
$A_{[X,Y]}\in \K\otimes\U^*$, $\nabla_XA_Y\in \K\otimes\U^*$.
Since 
 $$
R_{\tilde\nabla}(X,Y)=R_\nabla(X,Y)+(\nabla_XA_Y-\nabla_YA_X)+[A_X,A_Y]-A_{[X,Y]}
 $$
the post-composition with projection $\pi$ eliminates the entire freedom.
 \end{proof}

Note that the statement holds true in every case when we have a canonical choice of
a complementary subbundle to $\K\subset TM$ to be identified with the screen bundle,
for instance it works under mild nondegeneracy assumptions on the intrinsic torsion $Q_1$.

The intrinsic curvature is the second fundamental invariant tensor field and a source of scalar differential invariants  (by doing full contractions with an invariant frame, such as  $V_i$ or $\xi_i$ for $i\leq d$ from Section \ref{S3}). 

 \begin{remark}
Proposition  \ref{prop:2Dlocal} (and \ref{prop:2Dglobal}) does not straightforwardly generalize for $d>2$. Actually, the Hilbert polynomial $h_{d,2} = \tfrac{1}{12}d^4+ \tfrac12d^3+ \tfrac{5}{12}d^2-d$ has degree 4, while the number of generators used there, namely the structure functions $c_{ij}^l=-\hat d\omega^l(V_i,V_j)$, is $\tfrac{1}{2}d^2(d+1)$, which has cubic growth in $d$.
The number of components of the intrinsic curvature (in an invariant frame) is $\tfrac{1}{12}d^2(d^2-1)$. 
Since $\tfrac{1}{12}d^2(d^2-1)+\tfrac{1}{2}d^2(d+1)=h_{d,2}+d$, these may suffice to generate all second order invariants. 
However, the intrinsic curvature is defined through the second-order 1-form $dI_1^1$, whence some curvature components are of third order. 
One still has to verify if there are enough independent
remaining invariants of the second order.
 \end{remark}

\subsection{Symmetry bounds and examples}

The formal theory of PDEs helps identify the maximal symmetry of the structure. 
The affine characteristic variety (over $\C$) consists of projections to the second component of the set of rank one elements in $\g\subset\T\otimes\T^*$, so it is $\op{Char}(\g)=\U^*$. This has dimension $d$ and degree 1, hence 
the general solution depends on at most $1$ functions of $d$ arguments.

This latter is indeed realizable for the flat model \eqref{flat}, where the symmetries are
 $$
f(x^1,\dots,x^d)\p_t,\quad \p_{x^k},\quad x^i\p_{x^j}- x^j\p_{x^i}\quad (1\leq k\leq d,1\leq i<j\leq d).
 $$
In fact, this largest symmetry can only happen for the flat model. Indeed, 
the intrinsic torsion $q_1$ controls the size of the symmetry $f\p_t+\dots$, 
where $\langle\p_t\rangle=\K$ is the null line direction internal to the structure. 

In the next examples of a Carrollian structure we use adapted coordinates $(x,t)$, where $K=\p_t$. We first consider dimension 3 ($d=2$), where 
the maximal size symmetry depends on 1 function of 2 arguments.

 \begin{example}\label{conflat}
The submaximal symmetry is achieved for
 $$
g=e^{\lambda t}\bigl((dx^1)^2+(dx^2)^2\bigr).
 $$
Its symmetry algebra $\mathfrak{s}$ consists of the following vector fields
 $$
\lambda f_{x^2}\p_{x^1} +\lambda f_{x^1}\p_{x^2} -2f_{x^1x^2}\p_t,
 $$
where $f=f(x^1,x^2)$ is a harmonic function. 
The algebra $\mathfrak{s}$ is parametrized by 2 functions of 1 argument,
and contains the subalgebra
$\op{Euc}(2)=\mathfrak{so}(2)\ltimes\R^2$ generated by functions $f\in\{x^1,x^2,(x^1)^2-(x^2)^2\}$.
 \end{example}

 \begin{example}
The Carrollian structure 
 $$
g=e^{\lambda t}(dx^1)^2+(dx^2)^2
 $$
has nondegenerate torsion. Its symmetry algebra $\mathfrak{s}$ consists of the following vector fields
 $$
\lambda f\p_{x^1} +c\,\p_{x^2} -2f'\p_t,
 $$
where $f=f(x^1)$ is arbitrary function. 
Thus $\mathfrak{s}$ is parametrized by 1 functions of 1 argument.
 \end{example}

 \begin{example}\label{exlmbda12}
For nonzero $\lambda_1\neq\lambda_2$ the Carrollian structure 
 $$
g=e^{\lambda_1t}(dx^1)^2+e^{\lambda_2t}(dx^2)^2
 $$
has strongly nondegenerate torsion. Its symmetry algebra $\mathfrak{s}$ generated by
 $$
\p_{x^1},\ \p_{x^2},\ \p_t-\lambda_1x^1\p_{x^1}-\lambda_2x^2\p_{x^2}.
 $$
This $\mathfrak{s}$ is parametrized by 3 constants.
 \end{example}

Now we give an implication of the concept in Definition \ref{ndgq1}
on the symmetry algebra $\mathfrak{s}$.

 \begin{theorem}
A Carrollian structure $(M,g,K)$ with strongly nondegenerate intrinsic torsion $Q_1$ 
is of finite type, that is, it has finite-dimensional symmetry.
In fact, $\dim\mathfrak{s}\leq d+1$, so if $(M,g,K)$ is homogeneous
then the symmetry acts locally simply transitively. 
 \end{theorem}

 \begin{proof}
Let $\omega^i$ be the orthonormal eigen-coframe (locally defined up to $\pm$), 
in which $g=\sum_{i=1}^d(\omega^i)^2$ and $q_1=\sum_{i=1}^d\lambda_i(\omega^i)^2$.
Let us choose a foliation $U$ of $M$ complementary to $K$ (does not matter which)
and let $e_i$ be the dual frame tangent to this $U$. We can decompose any
infinitesimal symmetry as $X=f\p_t-\sum_1^da^ie_i$, where $K=\p_t$ and
$a^1,\dots,a^d,f\in C^\infty(M)$. Denote $c^k_{ij}=-d\omega^k(e_i,e_j)$
the structure functions. 

Note that the symmetry of $(g,K)$ must preserve $\omega^i$ and $\lambda_i$.
Since $\p_t\omega^k=\tfrac12\lambda_k\omega^k$ (no summation) we have
(no summation by repeated index unless indicated)
 \begin{equation}\begin{split}\label{spleq}
(L_X\omega^i)(e_j) &=(\iota_Xd\omega^i+d\iota_X\omega^i)(e_j)
=\tfrac12f\lambda_i\delta^i_j-\sum_k a^kc^i_{kj}-e_j(a^i)=0,\\
(L_X\omega^i)(\p_t) &=(\iota_Xd\omega^i+d\iota_X\omega^i)(\p_t)
=\tfrac12f\lambda_i a^i-\p_t(a^i)=0.
 \end{split}\end{equation}
Thus $\p_ta^i$, $e_j(a^i)$ and $\frac1{\lambda_i}e_i(a^i)-\frac1{\lambda_j}e_j(a^j)$
for $i\neq j$ express through 0-jets of $a^i$ only. 
Consequently all 2-jets of $a^i$ and 1-jet of $f$ are expressed through 
0-jets of $a^i$ and $f$. This implies the claim.
 \end{proof}

 \begin{remark}
An alternative argument for finite-type is the following. When we established
that $a^i$ depend only on finitely many parameters, it follows that if 
$\dim\mathfrak{s}=\infty$, then there exists a symmetry of type $f\p_t$ with 
$f\neq0$. However this implies $\lambda_i=0$, i.e.\ $Q_1=0$.
 \end{remark}

An example of homogeneous Carrollian structure with strongly nondegenerate intrinsic torsion is given by generalization of Example \ref{exlmbda12} as follows.

 \begin{example}
For nonzero $\lambda_i\neq\lambda_j$ the Carrollian structure 
 $$
g=\sum_{i=1}^de^{\lambda_it}(dx^i)^2
 $$
is homogeneous: its symmetry algebra $\mathfrak{s}$ is generated by
 $$
\p_{x^i},\ \p_t-\frac12 \sum_{i=1}^d\lambda_ix^i\p_{x^i}.
 $$
The trace invariants are given by 
$I_l^k = \mathrm{tr}\bigl((Q_l)^k\bigr) = \sum_{i=1}^d \lambda_i^{k l}$, 
for $l \geq 1$ and $1 \leq k \leq d$. 
 \end{example}

In this example, instead of coframe $dx^i$ we could choose a left-invariant coframe $\theta^i$ on Lie algebra $H$; this would give another homogeneous $M=H\times\R$. Next we consider the case of scalar $Q_1$, that is a degenerate 
conformally flat quadric, generalizing Example \ref{conflat}.

 \begin{example}
For $d>2$ the Carrollian structure 
 $$
g=e^{\lambda t}\sum_{i=1}^d(dx^i)^2
 $$
has symmetry algebra $\mathfrak{s}\simeq\mathfrak{so}(d+1,1)$ consisting of 
vector fields
 $$
\p_{x^i},\ x^i\p_{x^j}-x^j\p_{x^i},\ \xi:=\sum x^i\p_{x^i}-2 \lambda^{-1}\p_t,\
-\frac12\sum(x^i)^2\p_{x^k}+x^k\xi.
 $$
This is the conformal symmetry algebra.
The difference with the case $d=2$ is that the conformal 
symmetry in that dimension is infinite-dimensional as we saw in Example \ref{conflat}. The trace invariants are given by $I_l^k = \lambda^{kl}\, d$.

 \end{example}

There are non-flat examples in higher dimensions $d>2$ with infinite symmetry, 
however, the largest of those we observed depends only on functions of 1 argument.

\section{Outlook}

In this paper we discussed differential invariants of Carrollian structures
from the viewpoint of intrinsic geometry. Rational absolute invariants are 
global, but are undefined for a proper Zariski closed subset in jets, which
corresponds to close-to-flat structures. To describe those structure a technique
of relative differential invariants, similar to application in \cite{MS},
can be developed. 

For example, the shear $\sigma$, introduced in \cite{NR}, is a complex-valued
relative invariant, vanishing of which is explored in \cite{BMN}. 
The absolute value $|\sigma|$ has the same weight as the other relative invariant
$\rho$ of \cite{NR}. It is important to understand the weight lattice of
scalar relative differential invariants in this problem.

Related is the invariant variational problem for Carrollian structures.
Note that every Carrollian spacetime $(M,g,K)$ has a canonical volume form:
given the volume form $\omega$ of $g$ on $\U$ and arbitrary Ehresmann connection
$\tau$, with $\tau(K)=1$, the top form $\Omega=\tau\wedge\omega$ is well-defined
(independent of the choice of Ehresmann connection). Thus an arbitrary 
invariant Lagrangian density $\lambda$, obtained from the generators
of the algebra of scalar absolute differential invariants, gives rise
to an invariant Lagrangian $L=\lambda\Omega$ and the corresponding action $S=\int L$.
Such field theories have been investigated in \cite{Bagchi} and our
approach leads to the most general actions governing Carrollian spacetimes.

\appendix

\section{Intrinsic curvature in other non-relativistic geometries}

In this appendix we generalize computations of fundamental invariants 
for Carrollian spacetimes to other structures. 
We refer to details on the formal theory of PDEs to \cite{S,KL1,BK}.

\subsection{Spencer cohomology for Galilean spacetimes}

Galilean and Carrollian spacetimes have the same abstract linear symmetry 
$\g\simeq\mathfrak{so}(U)\ltimes U$ but embedded differently into $\mathfrak{gl}(V)$ 
it gives different prolongations, spaces of invariants, etc.

The Galilean spacetime arises in the non-relativistic limit $c\to\infty$.
It is given by 1-form $\tau$ and a (positive definite) symmetric bi-vector 
$\gamma$ on the distribution $\U=\op{Ker}(\tau)\subset\T$. The flat model is:
 \begin{equation}\label{GS}
M=\R^{d+1}(x^1,\dots,x^d,t),\quad \tau=dt,\quad \gamma=\sum_1^d\p_{x^i}\cdot\p_{x^i}.
 \end{equation}
In this case again $\T\simeq\U\oplus\K$ but (contrary to Carrollian) 
$\U$ is the injective part of $\T$ while $\K=\T/\U$ is the projective part.
Simlarly, $\U^*$ is a projective part of $\T^*$ 
while $\K^*=\op{Ann}(\U)=\langle\tau\rangle$ is the injective part.

The space $\U$ is equipped with Riemannian metric, so we identify $\U^*=\U$ and 
$\mathfrak{so}(\U)=\Lambda^2\U^*$.
This latter is embedded in the projective part $\U\otimes\U^*$ of the injective part 
$\U\otimes\T^*\subset\mathfrak{gl}(\T)$, and we will again exploit such 
non-canonical summands in what follows.

The symbol of the Lie equation for symmetries (Lie algebra of the structure group) is 
$\g=\Lambda^2\U+\U\otimes\K^*\subset\U\otimes\T^*\subset\mathfrak{gl}(\T)$. 
Its prolongations are:
 $$
\g_k=\g^{(k-1)}=(\Lambda^2\U\otimes S^{k-1}\K^*)\oplus(\U\otimes S^k\K^*)
\subset \T\otimes S^k\T^*,\quad k>0,
 $$
and we complete this by $\g_0=\T$. Using the same notations 
as in Theorem \ref{Spencer1} we get:

 \begin{theorem}\label{thA1}
The non-trivial Spencer cohomology groups for the Galilean spacetime are the following:
 \begin{gather*}
H^{0,0}=\T,\quad H^{0,1}=S^2\U\oplus\T^*,\quad
H^{0,k}=\Lambda^k\T^*\ \text{ for }\ 2\leq k\leq d+1,\\
H^{1,2}=\Gamma(2\pi_2),\quad H^{1,k}=\Gamma(\pi_2+\pi_k)\ \text{ for }\ 2<k<d,\quad H^{1,d}=\Gamma(\pi_2).
 \end{gather*}
 \end{theorem}

In particular, the full curvature module $H^{\bullet,2}(\g)$ consists of 
the intrinsic torsion part $H^{0,2}$ and the intrinsic curvature part $H^{1,2}$.
The fundamental invariants are $d\tau\in H^{0,2}(\g)$ and an analog of the
curvature $R_g$ in $H^{1,2}(\g)$.
The system $\g$ is involutive after one prolongation: $H^{i,\bullet}=0$ for $i>1$.

 \begin{proof}
Cohomology of the $0$-th and $1$-st Spencer complexes are straightforward.
For the others we use the same approach as in the proof of Theorem \ref{Spencer1},
though it is a bit more involved. Consider $k$-th Spencer complex.
First we quotient by subcomplexes of exact two-sequences 
$(\U\otimes\Lambda^i\U^*\otimes S^{k-i-j}\K^*\otimes\Lambda^j\K^*,\delta)_{j=0,1}$
successively for $i=0,1,\dots,k-1$. 
Then we quotient the result by subcomplexes of exact two-sequences 
$(\Lambda^2\U\otimes\Lambda^i\U^*\otimes S^{k-i-j-1}\K^*\otimes\Lambda^j\K^*,\delta)_{j=0,1}$
successively for $i=0,1,\dots,k-2$. This results in the following two-term complex
 $$
0\to\Lambda^2\U\otimes\Lambda^{k-1}\U^*\to (\U\otimes\Lambda^k\U^*)\oplus(\K\otimes\Lambda^k\T^*) \to0.
 $$
The last summand is $H^{0,k}$, while the other terms contribute to $H^{1,k-1}$,
which is computed through the same $A_{d-1}$ module decomposition as in 
the proof of Theorem \ref{Spencer1}.
 \end{proof}

The affine characteristic variety $\op{Char}_\text{aff}=\K^*$ has both dimension and
degree equal to 1. The kernel bundle over it has fiber $\mathcal{K}_p=M(\g)_p$ 
at $p\in\K^*_\times$, where $M(\g)=\oplus_{k=0}^\infty\g_k^*$ is the symbolic module 
over the polynomial algebra $S\T^*$ and subscript $p$ means localization on $S\T^*_p$. 
In our case $\dim\mathcal{K}_p=d+\binom{d}2=\binom{d+1}2$.
Hence the general solution depends on at most $\binom{d+1}2$ functions of 1 argument.

This latter is indeed realizable for the flat model \eqref{GS}, where the symmetries are
 $$
\p_t,\quad f(t)\p_{x^k},\quad \psi_{ij}(t)(x^i\p_{x^j}- x^j\p_{x^i})\quad (1\leq k\leq d,1\leq i<j\leq d).
 $$
In fact, this largest symmetry can only happen for the flat model, the number of functions decreases otherwise.
In the case of nondegenerate torsion $\tau\wedge d\tau\neq0$ 
(sub-Riemannian metric on completely nonholonomic distribution) 
the structure becomes of finite type, so the symmetry algebra is finite-dimensional.

\subsection{Spencer cohomology for $D$-branes}

Note that in Galillean spacetime $\tau^2=dt\cdot dt$ is the metric on $\K$
and in Carrollian spacetime $K^2=\p_t\cdot\p_t$ is the metric on $\K^*$.
Thus there are metrics on both injective and projective parts of $\T$ 
(and hence also of $\T^*$).
The interpolating spacetime between the two is given by 
a subbundle $\K\subset\T=TM$ and the quotient bundle $\U=\T/\K$
(this notation is similar to Carrollian, but opposite to Gallilean),
as well as choices of metrics $\zeta$ and $h$ on these bundles 
(and hence on the duals). In other words, $\K$ is the injective part, while
$\U$ is projective part. Opposite for duals. The flat model is
 \begin{equation}\label{IS}
M=\R^{m+n}(x^1,\dots,x^m,y^1,\dots,y^n),\quad h=\sum_1^m dx^i\cdot dx^i,\quad \zeta=\sum_1^n \p_{y^i}\cdot\p_{y^i}.
 \end{equation}
The leaves of $\K$, also known as $D$-branes (obtained by contracting $\zeta$ 
with $\T^*$), are $\R^n(y^1,\dots,y^n)$ and the quotient space (space of leaves)
is $\R^m(x^1,\dots,x^m)$.

Both spaces $\K,\U$ are equipped with Riemannian metrics, so we identify 
$\K=\K^*$, $\U^*=\U$ as well as $\mathfrak{so}(\K)=\Lambda^2\K$, 
$\mathfrak{so}(\U)=\Lambda^2\U^*$, and the previous conventions.

The symbol of the Lie equation for symmetries (Lie algebra of the structure group) 
is $\g=\Lambda^2\K+\K\otimes\U^*+S^2\U^*\subset\mathfrak{gl}(\T)$. 
Its prolongations are:
 $$
\g_k=\g^{(k-1)}=(\Lambda^2\K\otimes S^{k-1}\U^*)\oplus(\K\otimes S^k\U^*)
\subset \T\otimes S^k\T^*,\quad k>1,
 $$
which we complete by $\g_0=\T$, $\g_1=\g$. 

Using the same arguments as in Theorem \ref{Spencer1} and Theorem \ref{thA1} we obtain:

 \begin{theorem}
The non-trivial Spencer cohomology groups for $\g$ are the following:
 \begin{gather*}
H^{0,0}=\T,\quad H^{0,1}=(S^2\K)\oplus(\K\otimes\U^*)\oplus(S^2\U^*),\\ 
H^{0,k}=(\Lambda^{k-1}\K\otimes S^2\U^*)\oplus(\Lambda^k\K\otimes\U^*)\ \text{ for }\ 1<k\leq n+1,\\
H^{1,k}=\Gamma_\K(\pi_2+\pi_k)\oplus
\sum_{i=0}^{k-2}\Lambda^i\K\otimes\Gamma_\U(\pi_2+\pi_{k-i}),
 \end{gather*}
where in the latter line $\Gamma_\U$ is the corresponding module over $A_{m-1}=\mathfrak{sl}(\U)$ and
similar $\Gamma_\K$ is over $A_{n-1}=\mathfrak{sl}(\K)$; we adopt the convention
$\Gamma_{A_q}(\pi_2+\pi_j)$ is $\Gamma_{A_q}(\pi_2)$ for $j=q$ and $0$ for $j>q$.
 \end{theorem}

 \begin{proof}
To compute the Spencer cohomology we quotient out exact subcomplexes
$(\K\otimes S^{k-i}\U^*\otimes\Lambda^{i-s}\U^*\otimes\Lambda^s\K^*,\delta)$ and
$(\Lambda^2\K\otimes S^{k-i-1}\U^*\otimes\Lambda^{i-s}\U^*\otimes\Lambda^s\K^*,\delta)$
obtain a two-term complex and proceed as before.
 \end{proof}

In particular, the full curvature module consists of the torsion and the curvature:
 $$
H^{\bullet,2}(\g)=(\K\otimes S^2\U^*)\oplus(\Lambda^2\K\otimes\U^*)\oplus
\Gamma_\K(2\pi_2)\oplus\Gamma_\U(2\pi_2).
 $$
The system $\g$ is involutive after one prolongation ($H^{i,\bullet}=0$ for $i>1$).

The affine characteristic variety is $\op{Char}_\text{aff}=\U^*$ of degree 1, 
its kernel bundle has rank $\binom{n+1}2$.
Hence the general solution depends on at most $\binom{n+1}2$ functions of $m$ arguments.

This latter is indeed realizable for the flat model \eqref{IS}, where the symmetries are
 $$
f_0(x^1,\dots,x^n)\p_{y^k},\quad \psi_{ij}(x^1,\dots,x^n)(y^i\p_{y^j}-y^j\p_{y^i}),\quad 
\p_{x^k},\quad x^i\p_{x^j}- x^j\p_{x^i}.
 $$
In fact, this largest symmetry can only happen for the flat model.

\bigskip

{\bf Acknowledgment.}
This work was supported by the UiT Aurora project MASCOT, COST Action CaLISTA CA21109, the TFS funded Lie-Størmer Center, 
and RCN project MatNorVeg: Algebra-Geometri-Analyse med anvendelser \#360114.

\vspace{-3pt}
 \hspace{-20pt} {\hbox to 12cm{ \hrulefill }}
\vspace{5pt}

{\footnotesize \hspace{-10pt}\textsc{
Department of Mathematics and Statistics, UiT the Arctic University of Norway, 
Troms\o\ 9037, Norway}.

\hspace{-10pt} E-mails: \quad boris.kruglikov@uit.no,\quad eivind.schneider@uit.no,\quad wijnand.s.steneker@uit.no}
\vspace{-10pt}


\begin{thebibliography}{50}

\bibitem{ABL}
A.\ Ashtekar, C.\ Beetle, J.\ Lewandowski,
{\it Geometry of generic isolated horizons}, Class.\ Quantum Grav.\ {\bf 19}, 1195 (2002).

\bibitem{Bagchi}
A. Bagchi, R. Basu, A. Mehra, P. Nandi, {\it Field theories on null manifolds}, J. High Energ. Phys. {\bf 2020}, 141 (2020).

\bibitem{BM}
X.\ Bekaert, K.\ Morand, {\it Connections and dynamical trajectories 
in generalised Newton-Cartan gravity. II. An ambient perspective},
J.\ Math.\ Phys.\ {\bf 59}, 072503 (2018).

\bibitem{Blitz}
S.\ Blitz, D.\ McNutt, {\it Horizons that gyre and gimble: a differential characterization of null hypersurfaces}, Eur. Phys. J. C {\bf 84}, 561 (2024). 

\bibitem{BMN}
S.\ Blitz, D.\ McNutt, P.\ Nurowski, 
{\it Unique Carrollian manifolds emerging from Einstein spacetimes},
Class.\ Quantum Grav.\ {\bf 42}, 075006 (2025).

\bibitem{CST}
P.\,T.\ Chrusciel, S.\,J.\ Szybka, P.\ Tod, {\it Towards a classification 
of vacuum near-horizons geometries}, Class.\ Quantum Grav.\ {\bf 35}, 015002 (2018).

\bibitem{Ciambelli} 
L.\ Ciambelli, R.\,G.\ Leigh, C.\ Marteau, P.\ Marios Petropoulos, {\it Carroll Structures, Null Geometry and Conformal Isometries}, Phys. Rev. D {\bf 100}, 046010 (2019). 

\bibitem{CP}
L.\ Ciambelli, P.\ Jai-akson, {\it Foundations of Carrollian geometry},
Physics Reports {\bf 1188}, 1--51 (2026).

\bibitem{CG}
S.\ Curry, A.\,R.\ Gover, {\it An Introduction to Conformal Geometry and Tractor 
Calculus, with a view to Applications in General Relativity},
Asymptotic Analysis in General Relativity, 86--170, Cambridge Univ.\ Press (2017).

\bibitem{Duval} 
C. Duval, G.W. Gibbons, P.A. Horvathy, P.M. Zhang, {\it Carroll versus Newton and Galilei: two dual non-Einsteinian concepts of time}, Class. Quantum Grav. {\bf 31}, 085016 (2014). 

\bibitem{Duval2} 
C. Duval, G.W. Gibbons, P.A. Horvathy, {\it Conformal Carroll groups and BMS
symmetry}, Class. Quantum Grav. {\bf 31}, 092001 (2014). 

\bibitem{Figueroa} 
J.\ Figueroa-O'Farrill, {\it On the intrinsic torsion of spacetime structures}, arXiv:2009.01948 (2020). 

\bibitem{Figueroa1} 
J.\ Figueroa-O'Farrill, {\it Non-Lorentzian spacetimes}, Diff.\ Geom.\ Appl.\ 
{\bf 82}, 101894 (2022).

\bibitem{Figueroa2} 
J. Figueroa-O'Farrill, E.\ Have, S.\ Prohazka, J.\ Salzer, 
{\it Carrollian and celestial spaces at infinity}, J.\ High Energ.\ Phys.\ 
{\bf 2022}, 7 (2022).

\bibitem{FLTC}
A.\ Fino, T.\ Leistner, A.\ Taghavi-Chabert, {\it Optical geometries},
Ann.\ Sc.\ Norm.\ Super.\ Pisa Cl.\ Sci.\ {\bf 26}, no.\ 1, 341--396 (2025).

\bibitem{Gibbons}
G.\,W.\ Gibbons, {\it The Ashtekar-Hansen universal structure at spatial infinity 
is weakly pseudo-Carrollian}, arXiv:1902.09170 (2019).

\bibitem{Herfray}
Y.\ Herfray, {\it Carrollian manifolds and null infinity: a view from Cartan geometry}, Class.\ Quantum Grav.\ {\bf 39}, 215005 (2022).


\bibitem{BK}
B.~Kruglikov, {\it Poincare function for moduli of differential-geometric structures}, Moscow Math.\ Journ.\ {\bf 19}, no.4, 761--788 (2019).

\bibitem{KL1}
B.~Kruglikov, V.~Lychagin, {\it Geometry of differential equations},
in D.~Krupka, D.~Saunders (eds.), Handbook of Global Analysis, 725--771, Elsevier (2008).

\bibitem{KL2}
B.~Kruglikov, V.~Lychagin, {\it The global Lie-Tresse theorem}, Selecta Math.\ {\bf 22}, 1357--411, (2016).

\bibitem{KS}
B.\ Kruglikov, D.\ McNutt, E.\ Schneider, 
{\it Differential invariants of Kundt waves},
Class.\ Quantum Grav.\ {\bf 36}, 155011 (2019).


\bibitem{MS}
D.\ McNutt, E.\ Schneider, {\it Detecting horizons of symmetric black holes 
using relative differential invariants}, Class.\ Quantum Grav.\ {\bf 42}, 105011 (2025).

\bibitem{Morand}
K.\ Morand, {\it Embedding Galilean and Carrollian geometries. I. Gravitational waves}, J.\ Math.\ Phys.\ {\bf 61}, 082502 (2020).

\bibitem{NR}
P.\ Nurowski, D.\,C.\ Robinson, {\it Intrinsic geometry of a null hypersurface},
Class.\ Quantum Grav.\ {\bf 17}, 4065--4084 (2000).

\bibitem{O}
P.\ Olver, {\it Equivalence, Invariants, and Symmetry}, Cambridge University Press (1995).

\bibitem{PLJ}
T.\ Pawlowski, J.\ Lewandowski, J.\ Jezierski, {\it Spacetimes foliated by Killing horizons}, Class.\ Quantum Grav.\ {\bf 21}, 1237 (2004). 

\bibitem{S}
S.\ Sternberg, {\it Lectures on differential geometry}, N.J., Prentice-Hall (1964).

\end{thebibliography}
\end{document}